\documentclass[final,12pt]{msml2022} 

\title[MURANA: Stochastic Variance-Reduced Optimization]{MURANA: A Generic Framework for\\Stochastic Variance-Reduced Optimization}
\usepackage{times}

 \msmlauthor{\Name{Laurent Condat} \Email{laurent.condat@kaust.edu.sa}\and
  \Name{Peter Richt\'arik} \\
  \addr King Abdullah University of Science and Technology (KAUST), Thuwal 23955-6900, Saudi Arabia}
  
\makeatletter
 \let\Ginclude@graphics\@org@Ginclude@graphics
\makeatother

\usepackage{amsmath, amsfonts,amssymb}
\usepackage{algorithm}
\usepackage{algorithmic}
\usepackage{mathtools}
\newcommand{\eqdef}{\coloneqq}
\usepackage{caption}

\DeclareMathOperator*{\argmin}{arg\,min}
\DeclareMathOperator*{\minimize}{minimize}
\newcommand{\prox}{\mathrm{prox}}
\newcommand{\mmuF}{\mu}
\newcommand{\oma}{\omega_{\mathrm{av}}}
\newcommand{\ac}{a}
\newcommand{\gv}{v} 
\newcommand{\bc}{b} 
\newcommand{\omegac}{\omega_{\mathcal{C}}} 
\newcommand{\chic}{\omega_{\mathcal{U}}} 
\newcommand{\nuc}{\omega_{\mathcal{R}}}

\newcommand{\sqnorm}[1]{\left\| #1 \right\|^2}
\newcommand{\Exp}[1]{\mathbb{E}\!\left[ #1 \right]}

\begin{document}

\maketitle

\begin{abstract}%
 We propose a generic variance-reduced algorithm, which we call MUltiple RANdomized Algorithm (MURANA), for minimizing a sum of several smooth functions plus a regularizer, in a sequential or distributed manner. Our method is formulated with general stochastic operators, which allow us to model various strategies for reducing the computational complexity. For example, MURANA supports sparse activation of the gradients, and also reduction of the communication load via compression of the update vectors. This versatility allows MURANA to cover many existing randomization mechanisms within a unified framework, which also makes it possible to design new methods as special cases. 
\end{abstract}


\section{Introduction}

We consider the estimation of the model $x^\star\in \mathbb{R}^d$, for some $d\geq 1$, arising as the solution of the optimization problem
\begin{equation}
\minimize_{x\in\mathbb{R}^d} \left( R(x) + \frac{1}{M}\sum_{m=1}^M F_m(x)\right),\label{eqa1}
\end{equation}
for some $M\geq 1$, where each convex function $F_m$ is $L$-smooth, for some $L>0$, i.e.\ $\frac{1}{L}\nabla F_m$ is nonexpansive, and $R:\mathbb{R}^d\rightarrow \mathbb{R}\cup\{+\infty\}$ is a  proper, closed, convex function \citep{bau17}, whose proximity operator 
\begin{equation*}  
\prox_{\gamma R} : w \mapsto \argmin_{x\in \mathbb{R}^d} \left( \gamma R(x)+\frac{1}{2}\sqnorm{x-w} \right)
\end{equation*} 
is easy to compute, for any $\gamma>0$ \citep{par14,con19}.  We introduce 
\begin{equation*}  
F\eqdef \frac{1}{M}\sum_{m=1}^M F_m
\end{equation*} 
and we suppose that $F$ is 
$\mmuF$-strongly convex, for some 
$\mmuF>0$, i.e.\ $F-\frac{\mmuF}{2}\sqnorm{\cdot}$ is convex.
Since the problem \eqref{eqa1} is strongly convex, $x^\star$ exists and is unique.

 In a distributed client-server setting, $M$ is the number of parallel computing nodes, with an additional master node communicating with these $M$ nodes.  Communication between the master and nodes is often the bottleneck, so that it is desirable to reduce the amount of communicated information, in comparison with the baseline approach, where vectors of $\mathbb{R}^d$ are sent back and forth at every iteration. 
 
In a non-distributed setting, $M$ is, for instance, the number of data points contributing to some training task; it is then desirable to avoid scanning the entire dataset at every iteration.

\subsection{Randomized optimization algorithms} To formulate our algorithms, 
we will make use of several sources of randomness of the form 
\begin{equation}
d^k=\mathcal{C}^k\big(\nabla F(x^k)-h^k\big), \label{eq1}
\end{equation}
where $k$ is the iteration counter, $x^k\in\mathbb{R}^d$ is the model estimate converging to the desired solution $x^\star$,  $h^k$ is a control variate converging to $\nabla F(x^\star)$, and $\mathcal{C}^k(\gv)$ is a shorthand notation to denote a random realization of a stochastic process with expectation $\gv$, so that $\mathcal{C}^k(\gv)$ is a random unbiased estimate of the vector $\gv\in\mathbb{R}^d$. Although we adopt this notation as if $\mathcal{C}^k$ were a random operator, its argument $\gv$ does not always have to be known or computed. For instance, if 
\begin{equation*}
\mathcal{C}^k(\gv)= \begin{cases}\; \frac{1}{p} \gv  & \text{with probability} \;\; p \\ \;0  & \text{with probability} \;\; 1-p \end{cases},
\end{equation*}
 $\gv$ is not needed when the output is $0$. 
This means that in \eqref{eq1}, $\nabla F(x^k)$ is not computed in that case; this is the key reason why randomness makes it possible to decrease the overall complexity. The distribution of the random variable is not needed, and that is why we lighten the notations by omitting to write the underlying probability space structure. Indeed, we only need to know a constant $\omegac\geq 0$ such that, for every $\gv\in\mathbb{R}^d$,
\begin{equation}
\Exp{\big\| \mathcal{C}^k(\gv)-\gv \big\|^2 } \leq \omegac\|\gv\|^2,\label{eqc0}
\end{equation}
where the norm is the 2-norm and $\mathbb{E}[\cdot]$ denotes the expectation. Thus, if $\gv$ tends to $0$, not only does $\mathcal{C}^k(\gv)$ tend to $0$, but the variance tends to $0$ as well. Hence, in a step like in \eqref{eq1}, $d^k$ will converge to $0$ and everything will work out so that the algorithm converges to the exact solution $x^\star$.

That is, the proposed algorithm will be \textbf{variance reduced} \citep{gow20a}. In recent years, variance-reduced algorithms like SAGA \citep{def14} or SVRG \citep{joh13,zha13,xia14} have become the reference for finite-sum problems of the form \eqref{eqa1} since they converge to the exact solution but can be $M$ times faster than standard proximal gradient descent, which is typically a huge improvement. Variance reduction with the control variate $h^k$ is akin to an error-feedback mechanism, see \citet{con22} for a recent discussion on this relationship.
 
\subsection{Communication bottleneck in distributed and federated learning}
In the age of big data, there has been a shift towards distributed computations, and modern hardware increasingly relies on the power of uniting many parallel units into a single system. Training large machine learning models critically relies on distributed architectures. Typically,  the training data is distributed across several workers, which compute, in parallel, local updates of the model. These updates are then sent to a central server, which performs aggregation and then broadcasts the updated model back to the workers, to proceed with the next iteration. But communication of vectors between machines is typically much slower than computation, so \textbf{communication is the bottleneck}. 
This is even more true in the modern machine learning paradigm of \textbf{federated learning} \citep{kon16,mcm17,kai19,li20}, in which a global model is trained in a massively distributed manner over a network of heterogeneous devices, with a huge number of users involved in the learning task in a collaborative way. Communication can be costly, slow, intermittent and unreliable, and for that reason  the users ideally want to communicate the minimum amount of information. Moreover,  they also do not want to share their data for privacy reasons. 

Therefore, \textbf{compression} of the communicated vectors, using various sketching, sparsification, or quantization techniques \citep{ali17,wen17,wan18,alb20,bas20,dut20,sat20,xu21}, has become the approach of choice. In recent works \citep{tan19,liu20,phi20,gor20}, double, or bidirectional, compression is considered; that is, not only the vectors sent by the workers to the server, but also the model updates broadcast by the server to all workers, are compressed. 

Our proposed algorithm MURANA accommodates for model or bidirectional compression using the operators $\mathcal{R}^k$; see Section 2.1.

\subsection{A generic framework} Unbiased stochastic operators with conic variance, like in \eqref{eqc0}, allow to model a wide range of strategies: they can be used
\begin{itemize}
\item [(i)]  for \textbf{sampling}, i.e.\ to select a subset of functions whose gradient is computed at every iteration, like in SAGA or SVRG,  as mentioned above; 
\item [(ii)] for \textbf{compression}; in addition to the idea of communicating each vector only with some small probability, we can mention as example the \texttt{rand-k} operator, which sends $k$ out of $d$ elements, chosen at random and scaled by $\frac{d}{k}$, of its argument vector; 
\item [(iii)] to model \textbf{partial participation} in federated learning, with each user participating  in a fraction of the communication rounds only. 
\end{itemize}
That is why we formulate MURANA with this type of operators, which have all these applications, and many more.

\subsection{Contributions}  We propose MUltiple RANdomized Algorithm (MURANA) --  a generic template algorithm with several several sources of randomness  that can model a wide range of computation, communication reduction strategies, or both at the same time (e.g.\ by composition, see Proposition~\ref{prop2}). MURANA is variance reduced: it converges to the exact solution whatever the variance, 
which can be arbitrarily large. MURANA generalizes DIANA \citep{mis19,hor19} in several ways and encompasses SAGA \citep{def14} and loopless  SVRG \citep{hof15, kov20} as particular cases; we also give minibatch versions for them. Thus, our main contribution is to present these different algorithms within a unified framework, which allows us to derive convergence guarantees with weakened assumptions.

\section{Proposed framework: MURANA}

\subsection{Three sources of randomness}

We define $[M]\eqdef \{1,\ldots,M\}$. 
We first introduce the \textbf{first set of stochastic operators},  $\mathcal{C}_m^k$,
 for  every $k\geq 0$ and $m\in [M]$. In particular, we assume that  
there is a constant $\omegac\geq 0$ such that for every  $\gv\in\mathbb{R}^d$, 
\begin{equation}
\Exp{\mathcal{C}_m^k(\gv)}=\gv \quad\mbox{and} \quad \Exp { \sqnorm{\mathcal{C}_m^k(\gv)-\gv } } \leq \omegac \|\gv\|^2. \label{eqc1}
\end{equation}
 For every $(\gv,\gv' )\in(\mathbb{R}^d)^2$ and $(m,m' )\in [M]^2$, $\mathcal{C}_m^k(\gv)$ and $\mathcal{C}_{m'}^{k'}(\gv')$ at two different iteration indexes $k\neq k'$ are independent random variables. However, they can have different laws since only their first and second order statistics matter, as expressed in \eqref{eqc1}. Note that $\mathcal{C}_{m}^{k}(\gv)$ and $\mathcal{C}_{m'}^{k}(\gv')$ with $m\neq m'$ can be \textbf{dependent}, so  $\left(\mathcal{C}_{1}^{k}(\gv_1),\ldots, \mathcal{C}_{M}^{k}(\gv_M)\right)$ should be viewed as a whole joint random process; this is needed for sampling or partial participation, for instance, where $N<M$ indexes in $[M]$ are chosen at random; see Proposition~\ref{prop1} below.

Next, we introduce the \textbf{second set of stochastic operators},  $\mathcal{U}_m^k$, with same properties: for every $k\geq 0$, $m\in [M]$, $\gv\in\mathbb{R}^d$, 
\begin{equation}
\Exp{ \mathcal{U}_m^k(\gv) } = \gv \quad\mbox{and} \quad \Exp{ \sqnorm{ \mathcal{U}_m^k(\gv)-\gv } } \leq \chic \|\gv\|^2,\end{equation}
for some constant $\chic\geq 0$, and same dependence properties with respect to $m$ and $k$ as the  $\mathcal{C}_m^k$. $\mathcal{C}_m^k$ and $\mathcal{U}_{m'}^k$ can be dependent, and we will see this in the particular case of DIANA, where  $\mathcal{U}_{m}^k=\mathcal{C}_m^k$.

Finally, we introduce the \textbf{third set of stochastic operators}, $\mathcal{R}^k$, which will be applied to the model updates. For every $k\geq 0$ and $\gv\in\mathbb{R}^d$, 
\begin{equation}
\Exp{ \mathcal{R}^k(\gv) } = \gv \quad\mbox{and} \quad \Exp{ \sqnorm{ \mathcal{R}^k(\gv)-\gv } } \leq \nuc \|\gv\|^2, \label{eq6}
\end{equation}
for some constant  $\nuc\geq 0$. 
The operators $(\mathcal{R}^k)_{k\geq 0}$ are mutually independent and independent from all operators
$\mathcal{C}_m^{k'}$ and $\mathcal{U}_m^{k'}$. 

To analyze MURANA, we need to be more precise than just specifying the \textbf{marginal gain}  $\omegac$. So, we introduce the \textbf{average gain}
$\oma\geq 0$ and the \textbf{offset} $\zeta \in [0,\oma]$, such that, for every $k\geq 0$ and $\gv_m\in\mathbb{R}^d$, $m=1,\ldots,M$, 
\begin{equation}
\Exp{ \sqnorm{ \frac{1}{M}\sum_{m=1}^M \left(\mathcal{C}_m^k(\gv_m)-\gv_m\right)} } \leq \frac{\oma}{M} \sum_{m=1}^M \sqnorm{ \gv_m }-\zeta \sqnorm{ \frac{1}{M}\sum_{m=1}^M \gv_m}.\label{eqbo}
\end{equation}
We can assume that $\oma\leq \omegac$, since \eqref{eqbo} is satisfied with $\oma$ replaced by $\omegac$ and $\zeta$ by $0$, by convexity of the squared norm. In other words, 
without further knowledge, one can set $\oma= \omegac$ and $\zeta=0$. But the convergence rate will depend on $\oma$, not  $\omegac$, and the smaller $\oma$, the better. Thus, whenever $\oma$ is much smaller than $\omegac$, it is important to exploit this knowledge. In addition, having $\zeta>0$ allows to take larger stepsizes and have better constants in the convergence rates of the algorithms.

In particular, if the operators $(\mathcal{C}_m^k)_{m=1}^M$ are mutually independent, the variance of the sum is the sum of the variances, and we can set $\oma=\omegac/M$ and $\zeta=0$. Another case of interest is the sampling setting:

\begin{proposition}[Marginal and average gains of sampling]
\label{prop1}
Let $N\in[M]$. Consider that at every iteration $k$, a random subset $\Omega^k\subset [M]$ of size $N$  is chosen uniformly at random, and $\mathcal{C}_m^k$ is defined via 
\begin{equation*}
\mathcal{C}_m^k(\gv_m) \eqdef \begin{cases} \;\frac{M}{N} \gv_m & \text{ if }\;  m\in\Omega^k \\ 
\;0 & \text{ otherwise} \end{cases}.
\end{equation*}
 This is sometimes called $N$-nice sampling \citep{ric16,gow20}. Then \eqref{eqc1} is satisfied with $\omegac=\frac{M-N}{N}$ and  \eqref{eqbo} is satisfied with
\begin{equation}
\oma=\zeta=\frac{M-N}{N(M-1)}
\end{equation}
 (with $\oma=\zeta=0$ if $M=N=1$).
\end{proposition}

This property was proved in \citet{qia19}, but with different notations, so we give a new proof in Appendix~\ref{appena}, for sake of completeness.

Thus, in Proposition~\ref{prop1}, $\omegac$ can be as large as $M-1$, but we always have $\oma\leq 1$. \\

Furthermore, the stochastic operators can be composed, which makes it possible to combine random activation with respect to $m$ and compression of the vectors themselves, for instance:

\begin{proposition}[Marginal and average gains of composition]
\label{prop2}
Let $\mathcal{C}_m$ and $\mathcal{C}_m'$ be stochastic operators such that, for every $m\in[M]$ and
$\gv_m\in\mathbb{R}^d$, 
\begin{align*}
\Exp{ \mathcal{C}_m(\gv_m) } = \gv_m,\qquad  &\Exp{ \sqnorm{ \mathcal{C}_m(\gv_m)-\gv_m } } \leq \omegac \sqnorm{ \gv_m },\\
\Exp{ \mathcal{C}_m'(\gv_m) } =\gv_m, \qquad &\Exp{ \sqnorm{ \mathcal{C}_m'(\gv_m)-\gv_m } } \leq \omegac' \sqnorm{ \gv_m },\\
\Exp{ \sqnorm{ \frac{1}{M}\sum_{m=1}^M \left(\mathcal{C}_m'(\gv_m)-\gv_m\right) } } \leq {}&\frac{\oma'}{M} \sum_{m=1}^M \sqnorm{ \gv_m }-\zeta' \sqnorm{ \frac{1}{M}\sum_{m=1}^M \gv_m},
\end{align*}
for some $\omegac\geq 0$, $\omegac'\geq 0$, $\oma'\geq 0$, $\zeta'\geq 0$.
Then for every $m\in[M]$ and
$\gv_m\in\mathbb{R}^d$,
\begin{align}
\Exp{ \mathcal{C}_m'(\mathcal{C}_m(\gv_m)) } &=\gv_m,\\
\Exp{ \sqnorm{ \mathcal{C}_m'(\mathcal{C}_m(\gv_m))-\gv_m } } &\leq (\omegac+\omegac'+\omegac\omegac') \sqnorm{ \gv_m }.
\end{align}
Thus, the marginal gain of $\mathcal{C}_m' \circ \mathcal{C}_m$ is $\omegac+\omegac'+\omegac\omegac'$.

If, in addition, the  operators $(\mathcal{C}_m)_{m=1}^M$ are mutually independent, then for every 
$\gv_m\in\mathbb{R}^d$, $m=1,\ldots,M$, we get
\begin{align}
\Exp{ \sqnorm{ \frac{1}{M}\sum_{m=1}^M \Big(\mathcal{C}_m'\big(\mathcal{C}_m(\gv_m)\big)-\gv_m\Big) } } &\leq 
\left(\frac{\omegac}{M}(1-\zeta')+\oma'(1+\omegac)
\right)\frac{1}{M}\sum_{m=1}^M \sqnorm{ \gv_m }\notag\\
&\quad-\zeta'\sqnorm{ \frac{1}{M}\sum_{m=1}^M \gv_m }.
\end{align}
Thus, the average gain of the $\mathcal{C}_m' \circ \mathcal{C}_m$ in that case is $\frac{\omegac}{M}(1-\zeta')+\oma'(1+\omegac)$ and their offset is $\zeta'$.
\end{proposition}

\subsection{Proposed algorithms: MURANA and MURANA-D}

We propose the MUltiple RANdomized Algorithm (MURANA), described in Algorithm 1, as an abstract mathematical algorithm without regard to the execution architecture, or equivalently, as a sequential algorithm. We also explicitly write MURANA as a distributed algorithm in a client-server architecture, with explicit communication steps, as Algorithm 2, and call it MURANA-D.\medskip

If $\mathcal{U}_m^k=\mathcal{C}_m^k=\mathcal{R}^k=\mathrm{Id}$, where $\mathrm{Id}$ denotes the identity, and $\omegac=\chic=\oma=\nuc=0$, MURANA with $\lambda=\rho=1$ reverts to standard \textbf{proximal gradient descent}, which iterates: 
\begin{equation*}
x^{k+1}\eqdef\prox_{\gamma R}\left(x^{k}-\gamma \nabla F(x^k)\right).
\end{equation*}
 This baseline algorithm evaluates the full gradient $ \nabla F(x^k) = \frac{1}{M}\sum_{m=1}^M  \nabla F_m(x^k)$ at every iteration, which requires $M$ gradient calls. 
  If every gradient call has linear complexity $O(d)$, the complexity is $O(Md)$ per iteration, which is typically much too large.\medskip

Thus, the \textbf{three sources of randomness} in MURANA are typically used as follows: the operators $\mathcal{C}_m^k$ are used to save computation, by using much less than $M$, possibly even only 1, gradient calls per iteration, and/or decreasing the communication load by compressing the vectors sent by the nodes to the master for aggregation.  The operators $\mathcal{U}_m^k$ control the variance-reduction process, during which each variable $h_m^k$ learns the optimal gradient $\nabla F_m(x^\star)$ along the iterations, using the available computed information.  In a distributed setting, the operators $\mathcal{R}^k$ are used for compression during broadcast, in which the server communicates the model estimate to all nodes, at the beginning of every iteration.\medskip

When $\mathcal{U}_m^k=\mathcal{C}_m^k$ for every $m\in[M]$ and $k\geq 0$, we recover the recently proposed DIANA method of \citet{mis19,hor19} as a particular case of MURANA-D, but generalized here in several ways, see in Section~\ref{secdia}. 
In MURANA, we have \textbf{more degrees of freedom} than in DIANA: the stochastic gradient $d^{k+1} + h^k$, which is an unbiased estimate of $\nabla F(x^k)$ and is used to update the model $x^k$, is obtained  from the output of the operators $\mathcal{C}_m^k$, whereas the control variates $h_m^k$ learn the optimal gradients $\nabla F_m(x^\star)$ using the output of the operators $\mathcal{U}_m^k$. We can think of L-SVRG, see below in Section~\ref{seclv}, which has these two, different and decoupled,  mechanisms: the random choice of the activated gradient at every iteration and the random decision of taking a full gradient pass. 
Thus, MURANA is a versatile template algorithm, which covers many diverse tools spread across the literature of randomized optimization algorithms in a single umbrella.

\begin{figure*}[t]
\begin{minipage}{.48\textwidth}
\begin{algorithm}[H]
	\caption{MURANA (new)}
	\begin{algorithmic}[1]
		\STATE \textbf{input:}  parameters $\gamma>0$, $\lambda>0$, $\rho>0$, initial vectors $x^0\in\mathbb{R}^d$ and $h_m^{0}\in\mathbb{R}^d$, $m=1,\ldots,M$%
		\STATE   $h^{0}\eqdef\frac{1}{M}\sum_{m=1}^M h_m^{0}$
		\FOR{$k=0,1,\ldots$}{}
		\FOR{$m\in[M]$}{}
\STATE $d_m^{k+1}\eqdef\mathcal{C}_m^k\big(\nabla F_m(x^k)-h_m^{k}\big)$
\STATE $u_m^{k+1}\eqdef\mathcal{U}_m^k\big(\nabla F_m(x^k)-h_m^{k}\big)$
\STATE $h_m^{k+1}\eqdef h_m^{k}+\lambda u_m^{k+1}$
\ENDFOR
\STATE $d^{k+1}\eqdef\frac{1}{M}\sum_{m=1}^M  d_m^{k+1}$
\STATE $\tilde{x}^{k+1}\eqdef\prox_{\gamma R}\big(x^{k}-\gamma (h^k+d^{k+1})\big)$
\STATE $x^{k+1}\eqdef x^{k} + \rho\mathcal{R}^k(\tilde{x}^{k+1}-x^k)$
\STATE $h^{k+1}\eqdef h^k+\frac{\lambda}{M}\sum_{m=1}^M u_m^{k+1}$
		\ENDFOR
	\end{algorithmic}
\end{algorithm}\end{minipage}\ \ \ \ \ \begin{minipage}{.48\textwidth}
\begin{algorithm}[H]
	\caption{MURANA-D (new)}
	\begin{algorithmic}[1]
		\STATE \textbf{input:}  parameters $\gamma>0$, $\lambda>0$, $\rho>0$, initial vectors $x^0\in\mathbb{R}^d$ and  $h_m^{0}\in\mathbb{R}^d$, $m=1,\ldots,M$%
		\STATE  $h^{0}\eqdef\frac{1}{M}\sum_{m=1}^M h_m^{0}$, $r^0\eqdef 0$, $x^{-1}=x^0$
		\FOR{$k=0,1,\ldots$}{}
		\STATE  at master: broadcast $r^{k}$  to all nodes
		\FOR{$m\in[M]$, at nodes in parallel,}{}
		\STATE $x^k \eqdef x^{k-1}+\rho r^k$
\STATE $d_m^{k+1}\eqdef\mathcal{C}_m^k\big(\nabla F_m(x^k)-h_m^{k}\big)$
\STATE $u_m^{k+1}\eqdef\mathcal{U}_m^k\big(\nabla F_m(x^k)-h_m^{k}\big)$
\STATE $h_m^{k+1}\eqdef h_m^{k}+\lambda u_m^{k+1}$
\STATE convey $d_m^{k+1}$ and $u_m^{k+1}$ to master
\ENDFOR
\STATE at master:
\STATE $h^{k+1}\eqdef h^k+\frac{\lambda}{M}\sum_{m=1}^M u_m^{k+1}$
\STATE $d^{k+1}\eqdef\frac{1}{M}\sum_{m=1}^M  d_m^{k+1}$
\STATE $\tilde{x}^{k+1}\eqdef\prox_{\gamma R}\big(x^{k}-\gamma (h^k+d^{k+1})\big)$
\STATE $r^{k+1} \eqdef \mathcal{R}^k (\tilde{x}^{k+1}-x^k)$
\STATE $x^{k+1} \eqdef x^{k}+\rho r^{k+1}$
		\ENDFOR
	\end{algorithmic}
\end{algorithm}\end{minipage}%
\end{figure*}

\subsection{Convergence results}

We define $h_m^\star \eqdef  \nabla F_m(x^\star)$, $m=1,\ldots,M$, 
and we denote by $\kappa\eqdef L/\mmuF$ the conditioning of $F$.

\begin{theorem}[Linear convergence of MURANA]
\label{theo1}  In MURANA, suppose that 
$0<\lambda\leq \frac{1}{1+\chic}$ and $0<\rho\leq \frac{1}{1+\nuc}$, 
and set $\chic' \eqdef \frac{1}{\lambda}-1\geq \chic$ and $\nuc' \eqdef \frac{1}{\rho}-1\geq \nuc$. 
Choose $\bc>1$. Set $\ac\eqdef\max\big(1-(1+\bc)\zeta,0\big)$. 
Suppose that 
\begin{equation}
0<\gamma < \frac{2}{L}\frac{1}{\ac+(1+\bc)^2{\oma}}.
\end{equation}
Set 
$\eta
\eqdef 1-\gamma\left(\frac{2}{L} \frac{1}{\ac+(1+\bc)^2{\oma}}\right)^{-1} \in (0,1)
$. 
Define the Lyapunov function, for every $k\geq 0$,
\begin{equation}
\Psi^k \eqdef \big\|x^k-x^\star\big\|^2 +  (\bc^2+\bc)\gamma^2{\oma} \frac{1+\chic'}{1+\nuc'}\frac{1}{M}\sum_{m=1}^M \sqnorm{ h_m^k-h_m^\star }.\label{eqll1}
\end{equation}
Then, for every $k\geq 0$, we have \ 
$\Exp{ \Psi^{k} }  \leq c^k \Psi^0$,
 where 
\begin{equation}
c \eqdef 1-\min\left \{
\frac{2\gamma\eta\mmuF}{1+\nuc'},
  \frac{1-\bc^{-2}}{1+\chic'}\right \} <1.\label{eqcc1}
\end{equation}
\end{theorem}

Thus, MURANA converges linearly with rate $c$, in expectation; in particular, for every $k\geq 0$, $\Exp{ \sqnorm{ x^k-x^\star } } \leq c^k \Psi^0$. 
In addition, if MURANA is initialized with $h_m^0=\nabla F_m(x^0)$, for every $m\in[M]$, we have
\begin{equation}
\Psi^0\leq \bigg(1+(\bc^2+\bc)\gamma^2{\oma} \frac{1+\chic'}{1+\nuc'}L^2\bigg)\big\|x^0-x^\star\big\|^2.
\end{equation}\bigskip

The proof of Theorem~\ref{theo1} is deferred to Section~\ref{sec7}, for ease of reading.
\bigskip

In Theorem~\ref{theo1}, we have
\begin{equation*}
\gamma =  \frac{2(1-\eta)}{L}\frac{1}{\ac+(1+\bc)^2{\oma}},
\end{equation*}
so that
\begin{equation*}
2\gamma\eta\mmuF = 4(1-\eta)\eta \frac{\mmuF}{L}\frac{1}{\ac+(1+\bc)^2{\oma}}.
\end{equation*}
Maximizing this term, which appears in the rate $c$, with respect to $\eta$ yields $\eta=\frac{1}{2}$, so that the best choice for $\gamma$ is 
\begin{equation*}
\gamma =  \frac{1}{L}\frac{1}{\ac+(1+\bc)^2{\oma}}.
\end{equation*}

Thus, we can provide a simplified version of Theorem~\ref{theo1} as follows:

\begin{corollary}\label{cor1}
In MURANA, suppose that $\lambda= \frac{1}{1+\chic}$ and $\rho= \frac{1}{1+\nuc}$.  
Choose $\bc>1$. Set $\ac\eqdef\max\big(1-(1+\bc)\zeta,0\big)$.  Suppose that 
\begin{equation}
0<\gamma \leq \frac{1}{L}\frac{1}{\ac+(1+\bc)^2{\oma}}.\label{eqgammac1}
\end{equation}
Then, using $\Psi^k$ defined in \eqref{eqll1}, with $\chic'=\chic$ and $\nuc'=\nuc$, we have, 
for every $k\geq 0$, \linebreak
$\Exp{ \Psi^{k} }  \leq c^k \Psi^0$,
where 
\begin{equation}
c\eqdef 1-\min\left \{
\frac{\gamma \mmuF}{1+\nuc}
, \frac{1-\bc^{-2}}{1+\chic}\right\} <1.
\end{equation}
Therefore, if $\bc$ is fixed and $\gamma=\Theta\big(\frac{1}{L}\frac{1}{\ac+(1+\bc)^2{\oma}}\big)$, the asymptotic complexity of MURANA to achieve $\epsilon$-accuracy is
\begin{equation}
\mathcal{O}\Bigg(\Big(\kappa(1+\oma)(1+\nuc)+\chic \Big)\log\!\bigg(\frac{1}{\epsilon}\bigg)\Bigg)
\end{equation}
 iterations.\end{corollary}

\begin{proof}
The statements follow directly from the observation that, in the notations of Theorem~\ref{theo1}, the condition \eqref{eqgammac1} implies that $\eta\geq \frac{1}{2}$, so that $2\gamma\eta\mmuF \geq \gamma\mmuF$. 
\end{proof}

In the conditions of Corollary~\ref{cor1}, if we set $\gamma = \frac{1}{L}\frac{1}{\ac+(1+\bc)^2{\oma}}$,  we have:
 \begin{equation*}
c = 1-\min\left\{ 
\frac{1}{\kappa}\frac{1}{1+\nuc}\frac{1}{\ac+(1+\bc)^2{\oma}},  \frac{1-\bc^{-2}}{1+\chic} \right\}.
\end{equation*}
Thus, to balance the two constants $(1+\bc)^2$ and $1-\bc^{-2}$, we can choose 
$\bc=\sqrt{5}-1$, so that 
\begin{equation*}
c \leq  1-\min\left\{ 
\frac{1}{\kappa}\frac{1}{1+\nuc}\frac{1}{\ac+5{\oma}},
  \frac{1}{3}\frac{1}{1+\chic} \right\}.
\end{equation*}
Another choice is 
$\bc=\sqrt{6}-1$, so that 
\begin{equation*}
c \leq  1-\min\left\{ 
\frac{1}{\kappa}\frac{1}{1+\nuc}\frac{1}{\ac+6{\oma}},
  \frac{1}{2}\frac{1}{1+\chic} \right\}.
\end{equation*}

\section{Particular case: DIANA}\label{secdia}

When $\mathcal{U}^k_m=\mathcal{C}^k_m$, for every $k\geq 0$ and $m\in[M]$, and $\mathcal{R}^k=\mathrm{Id}$, MURANA-D reverts to DIANA, shown as Algorithm 3 (in the case $N=M$, i.e.\ full participation). 
DIANA was proposed by \citet{mis19} and generalized (with $R=0$) by
 \citet{hor19}. It was then further extended (still with $R=0$) to the case of  compression of the model during broadcast by \citet{gor20}, where it is called `DIANA with bi-directional quantization'; 
this corresponds to  $\mathcal{R}^k\neq \mathrm{Id}$ here, and we still call the algorithm DIANA in this case. An extension to $R\neq 0$ was made  by \citet{sigma_k}, who performed a unified analysis of a large class of non-variance-reduced and variance-reduced SGD-type methods under strong quasi-convexity. An analysis in the convex regime was performed by \citet{sigma_k-convex}.

However,  to date, DIANA was studied  for independent operators $\mathcal{C}^k_m$ only. Even in this case, our following results  are more general than existing ones. For instance, in Theorem~1 of \citet{hor19}, all functions $F_m$ are supposed to be strongly convex, whereas we only require their average $F$ to be strongly convex; this is a significantly weaker assumption.

Thus, we generalize  DIANA to arbitrary operators $\mathcal{C}^k_m$, to the presence of a regularizer $R$, and to possible randomization, or compression, of the model updates. 
As a direct application of Corollary~\ref{cor1} with $\chic=\omegac$, we have:

\begin{theorem}[Linear convergence of DIANA]
\label{theo2}
In DIANA, suppose that $\lambda= \frac{1}{1+\omegac}$ and $\rho= \frac{1}{1+\nuc}$. 
Choose $\bc>1$.  Set $\ac\eqdef\max\big(1-(1+\bc)\zeta,0\big)$.  Suppose that 
\begin{equation*}
0<\gamma \leq \frac{1}{L}\frac{1}{\ac+(1+\bc)^2{\oma}}.
\end{equation*}
Define the Lyapunov function, for every $k\geq 0$,
\begin{equation}
\Psi^k \eqdef \big\|x^k-x^\star\big\|^2 +  (\bc^2+\bc)\gamma^2{\oma} \frac{1+\omegac}{1+\nuc}\frac{1}{M}\sum_{m=1}^M \sqnorm{h_m^k-h_m^\star}.
\end{equation} 
Then, for every $k\geq 0$, we have \ 
$\Exp{ \Psi^{k} }  \leq c^k \Psi^0$,
where 
\begin{equation}
c \eqdef 1-\min\left \{ \frac{\gamma \mmuF}{1+\nuc}
, \frac{1-\bc^{-2}}{1+\omegac}\right\}<1.\label{eqcc2}
\end{equation}
Therefore, if $\bc$ is fixed and $\gamma=\Theta(\frac{1}{L}\frac{1}{\ac+(1+\bc)^2{\oma}})$, the  complexity of DIANA to achieve $\epsilon$-accuracy is 
\begin{equation}
\mathcal{O}\Bigg(\Big(\kappa(1+\oma)(1+\nuc)+\omegac \Big)\log\!\bigg(\frac{1}{\epsilon}\bigg)\Bigg)
\end{equation}
 iterations.
\end{theorem}

\begin{algorithm}[t]
	\caption{DIANA-PP (new)  (reverts to DIANA if $N=M)$}
	\begin{algorithmic}[1]
		\STATE \textbf{input:}  parameters $\gamma>0$, $\lambda>0$, $\rho>0$, participation level $N\in[M]$, initial vectors $x^0\in\mathbb{R}^d$ and $h_m^{0}\in\mathbb{R}^d$, $m=1,\ldots,M$
		\STATE  $h^{0}\eqdef\frac{1}{M}\sum_{m=1}^M h_m^{0}$,  $r^0\eqdef 0$, $x^{-1}=x^0$ 
		\FOR{$k=0,1,\ldots$}{}
		\STATE pick $\Omega^k\subset [M]$ of size $N$ uniformly at random 
		\STATE  at master: broadcast $r^{k}$  to all nodes 
		\FOR{$m\in \Omega_k$, at nodes in parallel,}{}
		\STATE $x^k \eqdef x^{k-1}+\rho r^k$
\STATE $d_m^{k+1}\eqdef\mathcal{C}_m^k\big(\nabla F_m(x^k)-h_m^{k}\big)$
\STATE $h_m^{k+1}\eqdef h_m^{k}+\lambda d_m^{k+1}$
\STATE convey $d_m^{k+1}$  to master
\ENDFOR
\FOR{$m\notin \Omega_k$, at nodes in parallel,}{}
\STATE $x^k \eqdef x^{k-1}+\rho r^k$
\STATE $h_m^{k+1}\eqdef h_m^{k}$
\ENDFOR
\STATE at master:
\STATE $d^{k+1}\eqdef\frac{1}{M}\sum_{m\in\Omega_k}  d_m^{k+1}$
\STATE $h^{k+1}\eqdef h^k+\lambda d^{k+1}$
\STATE $\tilde{x}^{k+1}\eqdef\prox_{\gamma R}\big(x^{k}-\gamma (h^k+d^{k+1})\big)$
\STATE $r^{k+1} \eqdef \mathcal{R}^k (\tilde{x}^{k+1}-x^k)$
\STATE $x^{k+1} \eqdef x^{k}+\rho r^{k+1}$
		\ENDFOR
	\end{algorithmic}
\end{algorithm}

\subsection{Partial participation  in DIANA} We make use of the possibility of having dependent stochastic operators and we use the composition of operators ${\mathcal{C}'}^k_{\!\!m} \circ \mathcal{C}^k_m$, like in Proposition~\ref{prop2}, with the ${\mathcal{C}'}^k_{\!\!m}$ being sampling operators like in Proposition~\ref{prop1}.
This yields DIANA-PP, shown as Algorithm 3. Since DIANA-PP is a particular case of DIANA with such composed operators, we can apply Theorem~\ref{theo2}, with $\omegac$, the marginal gain of the composed operators here, equal to $\omegac+\frac{M-N}{N}(1+\omegac)$, $\oma=\frac{\omegac}{M}+\frac{M-N}{N(M-1)}(1+\omegac)$, 
$\zeta = \frac{M-N}{N(M-1)}$:

\begin{theorem}[Linear convergence of DIANA-PP]
\label{theo3}
In DIANA-PP, suppose that the $(\mathcal{C}^k_m)_{m=1}^M$ are mutually independent and set $\oma\eqdef\frac{\omegac}{M}+\frac{M-N}{N(M-1)}(1+\omegac)$. Suppose that
$\lambda= \frac{N}{M}\frac{1}{1+\omegac}$ and $\rho=\frac{1}{1+\nuc}$. Choose $\bc>1$. 
Set $\ac\eqdef\max\Big(1-(1+\bc)\frac{M-N}{N(M-1)},0\Big)$. 
Suppose that 
\begin{equation*}
0<\gamma \leq \frac{1}{L}\frac{1}{\ac+(1+\bc)^2{\oma}}.
\end{equation*}
Define the Lyapunov function, for every $k\geq 0$,
\begin{equation}
\Psi^k \eqdef \big\|x^k-x^\star\big\|^2 +  (\bc^2+\bc)\gamma^2{\oma} \frac{1+\omegac}{1+\nuc}\frac{1}{N}\sum_{m=1}^M \sqnorm{h_m^k-h_m^\star}.
\end{equation}
Then, for every $k\geq 0$, we have \ 
$\Exp{ \Psi^{k} }  \leq c^k \Psi^0$,
where 
\begin{equation}c \eqdef 1-\min\left \{ 
\frac{\gamma \mmuF}{1+\nuc}, \frac{N}{M}\frac{1-\bc^{-2}}{1+\omegac}\right\}.
\end{equation}
Therefore,  if $\bc$ is fixed and $\gamma=\Theta(\frac{1}{L}\frac{1}{\ac+(1+\bc)^2{\oma}})$, the asymptotic complexity of DIANA-PP to achieve $\epsilon$-accuracy is 
\begin{equation}
  \mathcal{O}\Bigg(\left(\kappa\!\left(1+\frac{\omegac}{N}\right)\!(1+\nuc)+\frac{M}{N}(1+\omegac)\right)\log\!\bigg(\frac{1}{\epsilon}\bigg)\Bigg)
\end{equation}
 iterations.
\end{theorem}

To summarize, DIANA is the particular case of DIANA-PP with full participation, i.e. $N=M$. Its convergence with general, possibly dependent, operators $\mathcal{C}_m^k$, is established in Theorem~\ref{theo2}. DIANA-PP is more general than DIANA,  since it allows for partial participation, but its convergence is established in Theorem~\ref{theo3} only when the operators $(\mathcal{C}_m^k)_{m=1}^M$ are mutually independent.

\section{Particular case: SAGA}\label{secsag}

	\begin{algorithm}[t]
	\caption{Minibatch-SAGA (reverts to SAGA if $N=1$)}
	\begin{algorithmic}[1]
		\STATE \textbf{input:}  stepsize $\gamma>0$, sampling size $N\in[M]$, initial vectors $x^0\in\mathbb{R}^d$ and $h_m^{0}\in\mathbb{R}^d$, $m=1,\ldots,M$
		\STATE  $h^{0}\eqdef\frac{1}{M}\sum_{m=1}^M h_m^{0}$
		\FOR{$k=0,1,\ldots$}{}
		\STATE pick $\Omega^k\subset [M]$ of size $N$ uniformly at random
		\FOR{$m\in \Omega_k$}{}
\STATE $h_m^{k+1}\eqdef \nabla F_m(x^k)$
\ENDFOR
\FOR{$m\in[M]\backslash \Omega_k$}{}
\STATE $h_m^{k+1}\eqdef h_m^{k}$
\ENDFOR
\STATE $d^{k+1}\eqdef \frac{1}{N}\sum_{m\in\Omega^k}  (h_m^{k+1}-h_m^{k})$
\STATE $x^{k+1}\eqdef\prox_{\gamma R}\big(x^{k}-\gamma (h^k+d^{k+1})\big)$
\STATE $h^{k+1}\eqdef h^k+\frac{N}{M} d^{k+1}$
		\ENDFOR
	\end{algorithmic}
\end{algorithm}

When $\mathcal{U}^k_m=\mathcal{C}^k_m$, 
for every $k\geq 0$ and $m\in[M]$, and these operators are set as dependent sampling operators like in Proposition~\ref{prop1}, and $\mathcal{R}^k=\mathrm{Id}$, 
MURANA becomes Minibatch-SAGA, shown as Algorithm 4. We have $1+\omegac = \frac{M}{N}$, $\oma=\zeta=\frac{M-N}{N(M-1)}$, and we set $\lambda=\frac{1}{1+\omegac}= \frac{N}{M}$ and $\rho=1$. Minibatch-SAGA is SAGA  \citep{def14} if $N=1$ and proximal gradient descent if $N=M$, so Minibatch-SAGA interpolates between these two regimes for $1<N<M$. 
 This algorithm was called `minibatch SAGA with $\tau$-nice sampling' by \citet{gow20}, with their $\tau$ being our $N$, but studied only with $R=0$. It was called `q-SAGA' by \citet{hof15}  with their $q$ being our $N$, but studied only with all functions $F_m$ strongly convex. Thus, the following convergence results are new, to the best of our knowledge.
 
 As an application of Corollary~\ref{cor1}, we have:

\begin{theorem}[Linear convergence of Minibatch-SAGA]
\label{theo4}
Set $\oma\eqdef\frac{M-N}{N(M-1)}$ and
choose $\bc>1$. Set $\ac\eqdef\max\Big(1-(1+\bc)\frac{M-N}{N(M-1)},0\Big)$. 
 In Minibatch-SAGA, suppose that 
\begin{equation*}
0<\gamma \leq \frac{1}{L}\frac{1}{\ac+(1+\bc)^2{\oma}}.
\end{equation*}
Define the Lyapunov function, for every $k\geq 0$,
\begin{equation}
\Psi^k \eqdef \big\|x^k-x^\star\big\|^2 +  (\bc^2+\bc)\gamma^2{\oma}\frac{1}{N}\sum_{m=1}^M \sqnorm{h_m^k-h_m^\star}.
\end{equation}
Then, for every $k\geq 0$, we have \ 
$\Exp{ \Psi^{k} }  \leq c^k \Psi^0$,
 where 
 \begin{equation} 
c \eqdef 1-\min\left\{
\gamma \mmuF, \frac{N(1-\bc^{-2})}{M} \right\}<1.
\end{equation}
Therefore, if $\gamma=\Theta(\frac{1}{L})$, the asymptotic complexity of Minibatch-SAGA to achieve $\epsilon$-accuracy is $\mathcal{O}\left((\kappa+\frac{M}{N})\log(1/\epsilon)\right)$ iterations and $\mathcal{O}\big((N\kappa+M)\log(1/\epsilon)\big)$ gradient calls, since there are $N$ gradient calls per iteration.\end{theorem}

On a sequential machine without any memory access concern, $N=1$ is the best choice, but a larger $N$ might be better on more complex  architectures with memory caching strategies, or under more specific assumptions on the functions~\citep{gaz19,gow19}.\\

Let us state the convergence result for SAGA, as the particular case $N=1$ in Theorem~\ref{theo4}:

\begin{corollary}[linear convergence of SAGA]\label{cor2}
Choose $\bc>1$. 
 In SAGA, suppose that 
\begin{equation*}
0<\gamma \leq \frac{1}{L}\frac{1}{(1+\bc)^2}.
\end{equation*}
Define the Lyapunov function, for every $k\geq 0$,
\begin{equation}
\Psi^k \eqdef \big\|x^k-x^\star\big\|^2 +  (\bc^2+\bc)\gamma^2\sum_{m=1}^M \sqnorm{h_m^k-h_m^\star}.
\end{equation}
Then, for every $k\geq 0$, we have \ 
$\Exp{ \Psi^{k} }  \leq c^k \Psi^0$,
 where 
 \begin{equation} 
c \eqdef 1-\min\left\{
\gamma \mmuF, \frac{1-\bc^{-2}}{M} \right\}<1.
\end{equation}
Therefore, if $\gamma=\Theta(\frac{1}{L})$, the asymptotic complexity of SAGA to achieve $\epsilon$-accuracy is
\linebreak $\mathcal{O}\big((\kappa+M)\log(1/\epsilon)\big)$ iterations or gradient calls, since there is 1 gradient call per iteration.
\end{corollary}

In this result (and in the other ones as well), instead of first choosing $\bc$, one can choose $\gamma$ directly  and set $\bc$ accordingly, such that $\gamma= \frac{1}{L}\frac{1}{(1+\bc)^2}$. This yields:

\begin{corollary}[linear convergence of SAGA]\label{cor3}
 In SAGA, suppose that 
\begin{equation*}
0<\gamma < \frac{1}{4L}.
\end{equation*}
Set $\bc\eqdef\frac{1}{\sqrt{\gamma L}}-1$. 
Define the Lyapunov function, for every $k\geq 0$,
\begin{equation*}
\Psi^k \eqdef \big\|x^k-x^\star\big\|^2 +  (\bc^2+\bc)\gamma^2\sum_{m=1}^M \sqnorm{h_m^k-h_m^\star}.
\end{equation*}
Then, for every $k\geq 0$, we have \ 
$\Exp{ \Psi^{k} }  \leq c^k \Psi^0$, 
 where 
 \begin{equation*} 
c \eqdef 1-\min\left\{
\gamma \mmuF, \frac{1-\bc^{-2}}{M} \right\}<1.
\end{equation*}\end{corollary}

In Theorem 5.6 of \citet{bac21}, Bach gives a rate for SAGA with $\gamma=\frac{1}{4L}$ of $c=1-\min\big(\frac{3\mmuF}{16L},\frac{1}{3M}\big)$. Let us see how our results with the flexible constant $\bc$ make it possible to understand these constants and improve upon them. $\gamma=\frac{1}{4L}$ is not allowed in Corollaries \ref{cor2} and \ref{cor3}. So, let us invoke Theorem~\ref{theo1}, which is more general than Corollary~\ref{cor1}, with $\omegac=\chic=M-1$,  $\lambda=\frac{1}{M}$, $\nuc=0$, $\rho=1$, $\oma=\zeta=1$, $\ac=0$. We choose $\bc=\sqrt{5}-1$ and  $\gamma=\frac{1}{4L}$, so that $\eta=\frac{3}{8}$. Then we get a rate $c=1-\min\big(\frac{3\mmuF}{16L},\frac{1-\bc^{-2}}{M}\big)$, which is slightly better but almost the same as above, since $1-\bc^{-2}\approx 0.345 \approx \frac{1}{3}$. Now, keeping the same value of $\bc$ and choosing $\gamma=\frac{1}{L(1+\bc)^2}=\frac{1}{5L}$, Corollary~\ref{cor2} yields a rate $c=1-\min\big(\frac{\mmuF}{5L},\frac{1-\bc^{-2}}{M}\big)$, which is better, since $\frac{1}{5}>\frac{3}{16}$. On the other hand, choosing $\gamma=\frac{3}{16L}$ in Corollary~\ref{cor3} yields $\bc=\frac{4}{\sqrt{3}}-1$, so that $c=1-\min\big(\frac{3\mmuF}{16L},\frac{1-\bc^{-2}}{M}\big)$,  which is again better, since $1-\bc^{-2}\approx 0.41 > \frac{1}{3}$. Thus, $\gamma=\frac{3}{16L}$ and $\gamma=\frac{1}{5L}$, and every value in between, are uniformly better choices in SAGA than $\gamma=\frac{1}{4L}$, according to our analysis.

\section{Particular case: L-SVRG}\label{seclv}

Like SAGA, SVRG \citep{joh13,zha13} (sometimes called prox-SVRG \citep{xia14} if $R\neq 0$) is a variance-reduced 
randomized  algorithm, well suited to solve \eqref{eqa1}, since it can be up to $M$ times faster than proximal gradient descent. 

\begin{algorithm}[t]
	\caption{Minibatch-L-SVRG (reverts to L-SVRG if $N=1$)}
	\begin{algorithmic}[1]
		\STATE \textbf{input:}  parameter $\gamma>0$, sampling size $N\in[M]$, probability $p\in (0,1]$, initial vector $x^0\in\mathbb{R}^d$		
		\STATE $h^{0}\eqdef\frac{1}{M}\sum_{m=1}^M \nabla F_m(x^0)$, $y^0\eqdef x^0$
		\FOR{$k=0,1,\ldots$}{}
		\STATE Pick $\Omega^k\subset [M]$ of size $N$,  uniformly at random
\STATE $d^{k+1} \eqdef \frac{1}{N}\sum_{m\in\Omega^k} \big( \nabla F_m(x^k) - \nabla F_m(y^k)\big)$
\STATE $x^{k+1}\eqdef\prox_{\gamma R}\big(x^{k}-\gamma (h^k+d^{k+1})\big)$
\STATE $\displaystyle\mbox{Pick randomly }s^k\eqdef \begin{cases}\; 1 & \text{ with probability } \; p \\ \;0 & \text{ with probability } \; 1-p \end{cases}$
\IF{$s^k=1$}
\STATE $h^{k+1}\eqdef\frac{1}{M}\sum_{m=1}^M \nabla F_m(x^k)$
\STATE $y^{k+1}\eqdef x^k$
\ELSE
\STATE $h^{k+1}\eqdef h^k$, $y^{k+1} \eqdef y^k$
\ENDIF
		\ENDFOR
	\end{algorithmic}
\end{algorithm}

Recently, the loopless-SVRG (L-SVRG) algorithm was proposed  by \citet{hof15} and later rediscovered by \citet{kov20}. L-SVRG  is similar to SVRG, but with the outer loop of epochs replaced by a coin flip performed in each iteration, designed to trigger with a small probability, e.g. $1/M$, the computation of the full gradient of $F$. In comparison with SVRG, the analysis of L-SVRG is simpler and L-SVRG is more flexible; for instance,  there is no need to know $\mmuF$ to achieve the $\mathcal{O}\big((\kappa+M)\log(1/\epsilon)\big)$ complexity. 
In SVRG and L-SVRG, in addition to the full gradient passes computed once in a while, two gradients are computed at every iteration. A minibatch version  of L-SVRG, with $N$ instead of 1 gradients picked at every iteration, was called ``L-SVRG with  $\tau$-nice sampling'' by \citet{qia21}, see also \citet{seb19}; we call it Minibatch-L-SVRG, shown as Algorithm~5.

Minibatch-L-SVRG is a particular case of MURANA, with the $\mathcal{C}^k_m$, $m\in[M]$, set as dependent sampling operators like in Proposition~\ref{prop1}, and $\mathcal{R}^k=\mathrm{Id}$, $\rho=1$. Thus, like for Minibatch-SAGA, we have $1+\omegac = \frac{M}{N}$ and $\oma=\zeta=\frac{M-N}{N(M-1)}$.
Let $p \in (0,1]$.  The mappings $\mathcal{U}_m^k$ are all copies of the same random operator $\mathcal{U}^k$, defined by 
\begin{equation*}
\mathcal{U}^k(x) = \begin{cases} \;\frac{1}{p}x & \text{ with probability } \; p \\ 
\;0 & \text{ with probability } \; 1-p \end{cases}.
\end{equation*}
 We have $\chic=\frac{1-p}{p}$ and we set $\lambda=\frac{1}{1+\chic}=p$. We also set $h_m^k=\nabla F_m(y^k)$; these variables  are not stored in Minibatch-L-SVRG, but are computed upon request. 
 Hence, as an application of Corollary~\ref{cor1}, we get:

\begin{theorem}[Linear convergence of Minibatch-L-SVRG]
\label{theo5}
Set $\oma\eqdef\frac{M-N}{N(M-1)}$ and
choose $\bc>1$. Set $\ac\eqdef\max\Big(1-(1+\bc)\frac{M-N}{N(M-1)},0\Big)$. 
In Minibtach-L-SVRG, suppose that 
\begin{equation*}
0<\gamma \leq \frac{1}{L}\frac{1}{\ac+(1+\bc)^2{\oma}}.
\end{equation*}
Define the Lyapunov  function, for every $k\geq 0$,
\begin{equation}
\Psi^k \eqdef \big\|x^k-x^\star\big\|^2 +  (\bc^2+\bc)\gamma^2{\oma} \frac{1}{pM}\sum_{m=1}^M \sqnorm{h_m^k-h_m^\star}.\label{eqll5}
\end{equation}
Then, for every $k\geq 0$, we have \ 
$\Exp{ \Psi^{k} }  \leq c^k \Psi^0$,
 where
\begin{equation}
c \eqdef  1-\min\Big\{ 
\gamma \mmuF,   p(1-\bc^{-2}) \Big\}.\end{equation}
\end{theorem}

For instance, with $N=1$, $\bc=\sqrt{6}-1$, 
so that $\ac=0$, 
and $\gamma=\frac{1}{6L}$, we have $c\leq 1-\min \big(\frac{1}{6\kappa},p(1-\bc^{-2})\big)$; since $1-\bc^{-2} \approx 0.52 > \frac{1}{2}$, this is slightly better but very similar to the rate $1-\min (\frac{1}{6\kappa},\frac{p}{2})$ given in Theorem~5 of \citet{kov20}.

Therefore, if $\gamma=\Theta(\frac{1}{L})$, the asymptotic complexity of Minibatch-L-SVRG to achieve $\epsilon$-accuracy is $\mathcal{O}\left((\kappa+\frac{1}{p})\log(1/\epsilon)\right)$ iterations and $\mathcal{O}\left((N\kappa+pM\kappa+\frac{N}{p}+M)\log(1/\epsilon)\right)$ gradient calls, since there are 
$2N+pM$ gradient calls per iteration in expectation. This is the same as Minibatch-SAGA if $p=\Theta(\frac{N}{M})$.

\section{Particular case: ELVIRA (new)}

\begin{algorithm}[t]
	\caption{ELVIRA (new)}
	\begin{algorithmic}[1]
		\STATE \textbf{input:} stepsize $\gamma>0$, sampling size $N\in[M]$, probability $p\in (0,1]$, initial vector $x^0\in\mathbb{R}^d$
		\STATE $h^{0}\eqdef\frac{1}{M}\sum_{m=1}^M \nabla F_m(x^0)$, $y^0\eqdef x^0$
		\FOR{$k=0,1,\ldots$}{}
\STATE  $\mbox{Pick randomly }\displaystyle s^k\eqdef \begin{cases} \;1 & \text{ with probability } \; p \\ \;0 & \text{ with probability } \; 1-p \end{cases}$
\IF{$s^k=1$}
\STATE $h^{k+1}\eqdef\frac{1}{M}\sum_{m=1}^M \nabla F_m(x^k)$
\STATE $x^{k+1}\eqdef\prox_{\gamma R}\big(x^{k}-\gamma h^{k+1}\big)$
\STATE $y^{k+1}\eqdef x^k$
\ELSE
\STATE Pick $\Omega^k\subset [M]$ of size $N$, uniformly at random
\STATE $d^{k+1}\eqdef\frac{1}{N}\sum_{m\in\Omega^k}  \left(\nabla F_m(x^k) -\nabla F_m(y^k)\right)$
\STATE $x^{k+1}\eqdef\prox_{\gamma R}\big(x^{k}-\gamma (h^k+d^{k+1})\big)$
\STATE $h^{k+1}\eqdef h^k$, $y^{k+1} \eqdef y^k$
\ENDIF
		\ENDFOR
	\end{algorithmic}
\end{algorithm}

It  is a pity not to use the full gradient in L-SVRG to update $x^k$, when it is computed. And even with $p=1$, which means the full gradient computed at every iteration, L-SVRG does not revert to proximal gradient descent. We correct these drawbacks by proposing a new algorithm, called ELVIRA, shown as Algorithm~6. The novelty is that whenever a full gradient pass is computed, it is used just after to update the estimate $x^{k+1}$ of the solution. 

ELVIRA is a particular case of MURANA as follows: $\mathcal{R}^k=\mathrm{Id}$, $\rho=1$, and the $\mathcal{U}_m^k$ are set like in Minibatch-L-SVRG. The $\mathcal{C}_m^k$ depend on the $\mathcal{U}_m^k$ and are set as follows: if 
the full gradient is not computed, $\mathcal{C}_m^k$ are sampling operators like in Proposition~\ref{prop1}, Minibatch-L-SVRG and Minibatch-SAGA. Otherwise, the $\mathcal{C}_m^k$ are set to the identity. 

We have $\chic=\frac{1-p}{p}$ and we set $\lambda=\frac{1}{1+\chic}=p$. Moreover, $\oma=\zeta=\frac{M-N}{N(M-1)}(1-p)$. For instance, if $N=1$ and $p=\frac{1}{M}$, we have ${\oma}=\zeta=\frac{M-1}{M}$, instead of ${\oma}=\zeta=1$ with L-SVRG. Like in L-SVRG, we set $h_m^k=\nabla F_m(y^k)$; these variables  are not stored and are computed upon request. 

Hence, as an application of Corollary~\ref{cor1}, we get:

\begin{theorem}[Linear convergence of ELVIRA]
\label{theo6}
Set $\oma\eqdef\frac{M-N}{N(M-1)}(1-p)$ and
choose $\bc>1$.  Set $\ac\eqdef\max\Big(1-(1+\bc)(1-p)\frac{M-N}{N(M-1)},0\Big)$. 
In ELVIRA, suppose that
\begin{equation*}
0<\gamma \leq \frac{1}{L}\frac{1}{\ac+(1+\bc)^2{\oma}}.
\end{equation*}
Define the Lyapunov function, for every $k\geq 0$,
\begin{equation}
\Psi^k \eqdef \big\|x^k-x^\star\big\|^2 +  (\bc^2+\bc)\gamma^2{\oma} \frac{1}{pM}\sum_{m=1}^M \sqnorm{h_m^k-h_m^\star}.\label{eqll6}
\end{equation}
Then, for every $k\geq 0$, we have \ 
$\Exp{ \Psi^{k} }  \leq c^k \Psi^0$,
 where
\begin{equation}c \eqdef 1-\min\Big\{ 
\gamma \mmuF
,  p(1-\bc^{-2}) \Big\}.
\end{equation}\end{theorem}

For instance, with $N=1$, $\bc=\sqrt{6}-1$ 
and $\gamma=\frac{1}{6L}$, we have $c\leq 1-\min (\frac{1}{6\kappa},\frac{p}{2})$, like for L-SVRG. But for $N=1$ and a given $\bc>1$, the interval for $\gamma$ is slightly larger in ELVIRA than in L-SVRG. In other words, for a given $\gamma<\frac{1}{4L}$, one can choose a  larger value of $\bc$, yielding a smaller rate $c$.

Therefore, if $\gamma=\Theta(\frac{1}{L})$, the complexity of ELVIRA 
is $\mathcal{O}\left((\kappa+\frac{1}{p})\log(1/\epsilon)\right)$ iterations and $\mathcal{O}\left((N\kappa+pM\kappa+\frac{N}{p}+M)\log(1/\epsilon)\right)$ gradient calls, since there are $2N(1-p)+pM$ gradient calls per iteration in expectation. If in addition $p=\Theta(\frac{N}{M})$,
the complexity becomes $\mathcal{O}\left((\kappa+\frac{M}{N})\log(1/\epsilon)\right)$ iterations and $\mathcal{O}\big((N\kappa+M)\log(1/\epsilon)\big)$ gradient calls. 

So,  the asymptotic complexity of ELVIRA is the same as that of Minibatch-L-SVRG, and it has the same low-memory requirements.  But in practice, one can expect ELVIRA to be a bit faster, because its variance is strictly lower. This is illustrated by experiments in Appendix~\ref{secex}. 
ELVIRA reverts to proximal gradient descent if $p=1$ or $N=M$.

\section{Conclusion}

We have proposed a general framework for iterative algorithms minimizing a sum of functions by making calls to unbiased stochastic estimates of their gradients, and featuring variance-reduction mechanisms learning the optimal gradients. Our generic template algorithm MURANA allows us to study existing algorithms and design new ones within a unified analysis. Sampling among functions, compression of the vectors sent in both directions in distributed settings, e.g. by sparsification or quantization, as well as partial participation of the workers, which are of utmost importance in modern distributed and federated learning settings, are all features covered by our framework. In future work, we plan to exploit our findings to design new algorithms tailored to specific applications, and to investigate the following questions:
\begin{enumerate}
\item Can we relax the strong convexity assumption and still guarantee linear convergence of MURANA? For instance, in \citet{con22}, linear convergence of DIANA under a Kurdyka--{\L}ojasiewicz assumption has been proved.
\item Can we relax the unbiasedness assumption of the stochastic estimation processes? In \citet{con22}, a new class of possibly biased and random compressors is introduced, and linear convergence of DIANA with them is proved. 
\item  Can we prove last-iterate convergence as well as a sublinear rate for MURANA when the problem is convex but not strongly convex? And in the nonconvex setting? 
\item Can we extend the setting of stochastic gradients with variance-reduction mechanisms to other algorithms than proximal gradient descent, like primal--dual algorithms for optimization problems involving several nonsmooth terms~\citep{com21,con19,con22a}? An approach of this type has been proposed in \citet{sal20}, based on another proof technique with the Lagrangian gap, and it would be interesting to combine the two approaches. For instance, can we derive an algorithm like MURANA-D for decentralized optimization, and not only for the client-server setting, similar to the DESTROY algorithm in \citet{sal20}?
\end{enumerate}

	\bibliography{ieeeabrv,biblio2}

\begin{thebibliography}{42}
\providecommand{\natexlab}[1]{#1}
\providecommand{\url}[1]{\texttt{#1}}
\expandafter\ifx\csname urlstyle\endcsname\relax
  \providecommand{\doi}[1]{doi: #1}\else
  \providecommand{\doi}{doi: \begingroup \urlstyle{rm}\Url}\fi

\bibitem[Albasyoni et~al.(2020)Albasyoni, Safaryan, Condat, and
  {Richt\'arik}]{alb20}
A.~Albasyoni, M.~Safaryan, L.~Condat, and P.~{Richt\'arik}.
\newblock Optimal gradient compression for distributed and federated learning.
\newblock arXiv:2010.03246, 2020.

\bibitem[Alistarh et~al.(2017)Alistarh, Grubic, Li, Tomioka, and
  Vojnovic]{ali17}
D.~Alistarh, D.~Grubic, J.~Li, R.~Tomioka, and M.~Vojnovic.
\newblock {QSGD: Communication-efficient SGD via gradient quantization and
  encoding}.
\newblock In \emph{Proc. of 31st Conf. Neural Information Processing Systems
  (NIPS)}, pages 1709--1720, 2017.

\bibitem[Bach(2021)]{bac21}
F.~Bach.
\newblock Learning theory from first principles.
\newblock Draft of a book, version of Sept. 6, 2021, 2021.

\bibitem[Basu et~al.(2020)Basu, Data, Karakus, and Diggavi]{bas20}
D.~Basu, D.~Data, C.~Karakus, and S.~N. Diggavi.
\newblock {Qsparse-Local-SGD: Distributed SGD With Quantization,
  Sparsification, and Local Computations}.
\newblock \emph{IEEE Journal on Selected Areas in Information Theory},
  1\penalty0 (1):\penalty0 217--226, 2020.

\bibitem[Bauschke and Combettes(2017)]{bau17}
H.~H. Bauschke and P.~L. Combettes.
\newblock \emph{Convex Analysis and Monotone Operator Theory in Hilbert
  Spaces}.
\newblock Springer, New York, 2nd edition, 2017.

\bibitem[Combettes and Pesquet(2021)]{com21}
P.~L. Combettes and J.-C. Pesquet.
\newblock Fixed point strategies in data science.
\newblock \emph{{IEEE} Trans. Signal Process.}, 69:\penalty0 3878--3905, 2021.

\bibitem[Condat et~al.(2022{\natexlab{a}})Condat, Kitahara, Contreras, and
  Hirabayashi]{con19}
L.~Condat, D.~Kitahara, A.~Contreras, and A.~Hirabayashi.
\newblock Proximal splitting algorithms for convex optimization: A tour of
  recent advances, with new twists.
\newblock \emph{SIAM Review}, 2022{\natexlab{a}}.
\newblock to appear.

\bibitem[Condat et~al.(2022{\natexlab{b}})Condat, Malinovsky, and
  Richt{\'a}rik]{con22a}
L.~Condat, G.~Malinovsky, and P.~Richt{\'a}rik.
\newblock Distributed proximal splitting algorithms with rates and
  acceleration.
\newblock \emph{Frontiers in Signal Processing}, 1, January 2022{\natexlab{b}}.

\bibitem[Condat et~al.(2022{\natexlab{c}})Condat, Yi, and Richt{\'a}rik]{con22}
L.~Condat, K.~Yi, and P.~Richt{\'a}rik.
\newblock {EF-BV: A} unified theory of error feedback and variance reduction
  mechanisms for biased and unbiased compression in distributed optimization.
\newblock arXiv:2205.04180, 2022{\natexlab{c}}.

\bibitem[Defazio et~al.(2014)Defazio, Bach, and {Lacoste-Julien}]{def14}
A.~Defazio, F.~Bach, and S.~{Lacoste-Julien}.
\newblock {SAGA: A} fast incremental gradient method with support for
  non-strongly convex composite objectives.
\newblock In \emph{Proc. of 28th Conf. Neural Information Processing Systems
  (NIPS)}, pages 1646--1654, 2014.

\bibitem[Dutta et~al.(2020)Dutta, Bergou, Abdelmoniem, Ho, Sahu, Canini, and
  Kalnis]{dut20}
A.~Dutta, E.~H. Bergou, A.~M. Abdelmoniem, C.~Y. Ho, A.~N. Sahu, M.~Canini, and
  P.~Kalnis.
\newblock On the discrepancy between the theoretical analysis and practical
  implementations of compressed communication for distributed deep learning.
\newblock In \emph{Proc. of AAAI Conf. Artificial Intelligence}, pages
  3817--3824, 2020.

\bibitem[Gazagnadou et~al.(2019)Gazagnadou, Gower, and Salmon]{gaz19}
N.~Gazagnadou, R.~Gower, and J.~Salmon.
\newblock Optimal mini-batch and step sizes for {SAGA}.
\newblock In \emph{Proc. of 36th Int. Conf. Machine Learning (ICML)}, volume
  PMLR 97, pages 2142--2150, 2019.

\bibitem[Gorbunov et~al.(2020{\natexlab{a}})Gorbunov, Hanzely, and
  {Richt\'{a}rik}]{sigma_k}
E.~Gorbunov, F.~Hanzely, and P.~{Richt\'{a}rik}.
\newblock A unified theory of {SGD: Variance} reduction, sampling, quantization
  and coordinate descent.
\newblock In \emph{Proc. of 23rd Int. Conf. Artificial Intelligence and
  Statistics (AISTATS)}, 2020{\natexlab{a}}.

\bibitem[Gorbunov et~al.(2020{\natexlab{b}})Gorbunov, Kovalev, Makarenko, and
  {Richt\'arik}]{gor20}
E.~Gorbunov, D.~Kovalev, D.~Makarenko, and P.~{Richt\'arik}.
\newblock Linearly converging error compensated {SGD}.
\newblock In \emph{Proc. of 34th Conf. Neural Information Processing Systems
  (NeurIPS)}, 2020{\natexlab{b}}.

\bibitem[Gower et~al.(2019)Gower, Loizou, Qian, Sailanbayev, Shulgin, and
  Richt\'{a}rik]{gow19}
R.~M. Gower, N.~Loizou, X.~Qian, A.~Sailanbayev, E.~Shulgin, and
  P.~Richt\'{a}rik.
\newblock {SGD}: {G}eneral analysis and improved rates.
\newblock In \emph{Proc. of 36th Int. Conf. Machine Learning (ICML)}, volume
  PMLR 97, pages 5200--5209, 2019.

\bibitem[Gower et~al.(2020)Gower, Schmidt, Bach, and {Richt\'arik}]{gow20a}
R.~M. Gower, M.~Schmidt, F.~Bach, and P.~{Richt\'arik}.
\newblock Variance-reduced methods for machine learning.
\newblock \emph{Proc. of the IEEE}, 108\penalty0 (11):\penalty0 1968--1983,
  November 2020.

\bibitem[Gower et~al.(2021)Gower, {Richt\'arik}, and Bach]{gow20}
R.~M. Gower, P.~{Richt\'arik}, and F.~Bach.
\newblock Stochastic quasi-gradient methods: {V}ariance reduction via
  {J}acobian sketching.
\newblock \emph{Math. Program.}, 188:\penalty0 135--192, July 2021.

\bibitem[Hofmann et~al.(2015)Hofmann, Lucchi, {Lacoste-Julien}, and
  {McWilliams}]{hof15}
T.~Hofmann, A.~Lucchi, S.~{Lacoste-Julien}, and B.~{McWilliams}.
\newblock Variance reduced stochastic gradient descent with neighbors.
\newblock In \emph{Proc. of 29th Conf. Neural Information Processing Systems
  (NIPS)}, pages 1509--1519, 2015.

\bibitem[{Horv\'ath} et~al.(2019){Horv\'ath}, Kovalev, Mishchenko, Stich, and
  {Richt\'arik}]{hor19}
S.~{Horv\'ath}, D.~Kovalev, K.~Mishchenko, S.~Stich, and P.~{Richt\'arik}.
\newblock Stochastic distributed learning with gradient quantization and
  variance reduction.
\newblock arXiv:1904.05115, 2019.

\bibitem[Johnson and Zhang(2013)]{joh13}
R.~Johnson and T.~Zhang.
\newblock Accelerating stochastic gradient descent using predictive variance
  reduction.
\newblock In \emph{Proc. of 27th Conf. Neural Information Processing Systems
  (NIPS)}, pages 315--323, 2013.

\bibitem[Kairouz et~al.(2021)]{kai19}
P.~Kairouz et~al.
\newblock Advances and open problems in federated learning.
\newblock \emph{Foundations and Trends in Machine Learning}, 14\penalty0
  (1--2), 2021.

\bibitem[Khaled et~al.(2020)Khaled, Sebbouh, Loizou, M., and
  {Richt\'{a}rik}]{sigma_k-convex}
A.~Khaled, O.~Sebbouh, N.~Loizou, R.~M.~Gower M., and P.~{Richt\'{a}rik}.
\newblock Unified analysis of stochastic gradient methods for composite convex
  and smooth optimization.
\newblock arXiv:2006.11573, 2020.

\bibitem[Kone\v{c}n\'{y} et~al.(2016)Kone\v{c}n\'{y}, McMahan, Yu,
  Richt\'{a}rik, Suresh, and Bacon]{kon16}
J.~Kone\v{c}n\'{y}, H.~B. McMahan, F.~X. Yu, P.~Richt\'{a}rik, A.~T. Suresh,
  and D.~Bacon.
\newblock Federated learning: {S}trategies for improving communication
  efficiency.
\newblock Paper arXiv:1610.05492, presented at the NIPS Workshop on Private
  Multi-Party Machine Learning, 2016.

\bibitem[Kovalev et~al.(2020)Kovalev, {Horv\'ath}, and {Richt\'arik}]{kov20}
D.~Kovalev, S.~{Horv\'ath}, and P.~{Richt\'arik}.
\newblock Don't jump through hoops and remove those loops: {SVRG and Katyusha}
  are better without the outer loop.
\newblock In \emph{Proc. of 31st Int. Conf. Algorithmic Learning Theory (ALT)},
  volume PMLR 117, pages 451--467, 2020.

\bibitem[Li et~al.(2020)Li, Sahu, Talwalkar, and Smith]{li20}
T.~Li, A.~K. Sahu, A.~Talwalkar, and V.~Smith.
\newblock Federated learning: Challenges, methods, and future directions.
\newblock \emph{IEEE Signal Processing Magazine}, 3\penalty0 (37):\penalty0
  50--60, 2020.

\bibitem[Liu et~al.(2020)Liu, Li, Tang, and Yan]{liu20}
X.~Liu, Y.~Li, J.~Tang, and M.~Yan.
\newblock A double residual compression algorithm for efficient distributed
  learning.
\newblock In \emph{Proc. of 23rd Int. Conf. Artificial Intelligence and
  Statistics (AISTATS)}, volume PMLR 108, pages 133--143, 2020.

\bibitem[McMahan et~al.(2017)McMahan, Moore, Ramage, Hampson, and {Ag\"{u}era y
  Arcas}]{mcm17}
H.~Brendan McMahan, Eider Moore, Daniel Ramage, Seth Hampson, and Blaise
  {Ag\"{u}era y Arcas}.
\newblock Communication-efficient learning of deep networks from decentralized
  data.
\newblock In \emph{Proc. of 20th Int. Conf. Artificial Intelligence and
  Statistics (AISTATS)}, 2017.

\bibitem[Mishchenko et~al.(2019)Mishchenko, Gorbunov, {Tak\'a\v{c}}, and
  {Richt\'arik}]{mis19}
K.~Mishchenko, E.~Gorbunov, M.~{Tak\'a\v{c}}, and P.~{Richt\'arik}.
\newblock Distributed learning with compressed gradient differences.
\newblock arXiv:1901.09269, 2019.

\bibitem[Parikh and Boyd(2014)]{par14}
N.~Parikh and S.~Boyd.
\newblock Proximal algorithms.
\newblock \emph{Foundations and Trends in Optimization}, 3\penalty0
  (1):\penalty0 127--239, 2014.

\bibitem[Philippenko and Dieuleveut(2020)]{phi20}
C.~Philippenko and A.~Dieuleveut.
\newblock Bidirectional compression in heterogeneous settings for distributed
  or federated learning with partial participation: tight convergence
  guarantees.
\newblock arXiv:2006.14591, 2020.

\bibitem[Qian et~al.(2019)Qian, Sailanbayev, Mishchenko, and
  Richt{\'a}rik]{qia19}
X.~Qian, A.~Sailanbayev, K.~Mishchenko, and P.~Richt{\'a}rik.
\newblock {MISO} is making a comeback with better proofs and rates.
\newblock arXiv:1906.01474, June 2019.

\bibitem[Qian et~al.(2021)Qian, Qu, and {Richt\'arik}]{qia21}
X.~Qian, Z.~Qu, and P.~{Richt\'arik}.
\newblock {L-SVRG and L-Katyusha} with arbitrary sampling.
\newblock \emph{Journal of Machine Learning Research}, 22\penalty0
  (112):\penalty0 1--47, 2021.

\bibitem[{Richt\'arik} and {Tak\'a\v{c}}(2016)]{ric16}
P.~{Richt\'arik} and M.~{Tak\'a\v{c}}.
\newblock Parallel coordinate descent methods for big data optimization.
\newblock \emph{Math. Program.}, 156:\penalty0 433--484, 2016.

\bibitem[Salim et~al.(2022)Salim, Condat, Mishchenko, and {Richt\'arik}]{sal20}
A.~Salim, L.~Condat, K.~Mishchenko, and P.~{Richt\'arik}.
\newblock Dualize, split, randomize: Fast nonsmooth optimization algorithms.
\newblock \emph{Journal of Optimization Theory and Applications}, 2022.
\newblock to appear.

\bibitem[Sattler et~al.(2020)Sattler, Wiedemann, {K.-R. M\"uller}, and
  Samek]{sat20}
F.~Sattler, S.~Wiedemann, {K.-R. M\"uller}, and W.~Samek.
\newblock Robust and communication-efficient federated learning from non-i.i.d.
  data.
\newblock \emph{IEEE Trans. Neural Networks and Learning Systems}, 31\penalty0
  (9):\penalty0 3400--3413, 2020.

\bibitem[Sebbouh et~al.(2019)Sebbouh, Gazagnadou, Jelassi, Bach, and
  Gower]{seb19}
O.~Sebbouh, N.~Gazagnadou, S.~Jelassi, F.~Bach, and R.~Gower.
\newblock Towards closing the gap between the theory and practice of {SVRG}.
\newblock In \emph{Proc. of 33rd Conf. Neural Information Processing Systems
  (NeurIPS)}, 2019.

\bibitem[Tang et~al.(2019)Tang, Yu, Lian, Zhang, and Liu]{tan19}
H.~Tang, C.~Yu, X.~Lian, T.~Zhang, and J.~Liu.
\newblock Doublesqueeze: {P}arallel stochastic gradient descent with
  double-pass error-compensated compression.
\newblock In \emph{Proc. of Int. Conf. Machine Learning (ICML)}, pages
  6155--6165, 2019.

\bibitem[Wangni et~al.(2018)Wangni, Wang, Liu, and Zhang]{wan18}
J.~Wangni, J.~Wang, J.~Liu, and T.~Zhang.
\newblock Gradient sparsification for communication-efficient distributed
  optimization.
\newblock In \emph{Proc. of 32nd Conf. Neural Information Processing Systems
  (NeurIPS)}, pages 1306--1316, 2018.

\bibitem[Wen et~al.(2017)Wen, Xu, Yan, Wu, Wang, Chen, and Li]{wen17}
W.~Wen, C.~Xu, F.~Yan, C.~Wu, Y.~Wang, Y.~Chen, and H.~Li.
\newblock {TernGrad: Ternary} gradients to reduce communication in distributed
  deep learning.
\newblock In \emph{Proc. of 31st Conf. Neural Information Processing Systems
  (NIPS)}, pages 1509--1519, 2017.

\bibitem[Xiao and Zhang(2014)]{xia14}
L.~Xiao and T.~Zhang.
\newblock A proximal stochastic gradient method with progressive variance
  reduction.
\newblock \emph{SIAM J. Optim.}, 24\penalty0 (4):\penalty0 2057--2075, 2014.

\bibitem[Xu et~al.(2021)Xu, Ho, Abdelmoniem, Dutta, Bergou, Karatsenidis,
  Canini, and Kalnis]{xu21}
H.~Xu, C.-Y. Ho, A.~M. Abdelmoniem, A.~Dutta, E.~H. Bergou, K.~Karatsenidis,
  M.~Canini, and P.~Kalnis.
\newblock {GRACE: A compressed communication framework for distributed machine
  learning}.
\newblock In \emph{Proc. of 41st IEEE Int. Conf. Distributed Computing Systems
  (ICDCS)}, 2021.

\bibitem[Zhang et~al.(2013)Zhang, Mahdavi, and Jin]{zha13}
L.~Zhang, M.~Mahdavi, and R.~Jin.
\newblock Linear convergence with condition number independent access of full
  gradients.
\newblock In \emph{Proc. of 27th Conf. Neural Information Processing Systems
  (NIPS)}, 2013.

\end{thebibliography}
	
	\newpage

\appendix

\section{Proof of Proposition~\ref{prop1}}\label{appena}

The first statement with the value of $\omegac$ follows from
\[
\Exp{ \sqnorm{ \mathcal{C}_m^k(\gv)-\gv } } =\frac{N}{M}\left(\frac{M}{N}-1\right)^{\!2} \sqnorm{\gv_m} + \frac{M-N}{M}\sqnorm{\gv_m} 
=\frac{M-N}{N}\sqnorm{\gv_m}.
\]\
Let us establish the second statement with the values of $\oma$ and $\zeta$.
We start with the identity, where $\mathbb{E}_{\Omega^k}$ denotes expectation with respect to the random set $\Omega^k$:
\begin{align*}
\Exp{ \sqnorm{ \sum_{m=1}^M \big(\mathcal{C}^k_m(\gv_m)-\gv_m\big) } } &= \mathbb{E}_{\Omega^k}\left[ \sqnorm{ \sum_{m\in \Omega^k} \frac{M}{N}\gv_{m} 
- \sum_{m=1}^M  \gv_m } \right] \\
&= \frac{M^2}{N^2} \mathbb{E}_{\Omega^k}\left[ \sqnorm{ \sum_{m\in \Omega^k}\gv_{m} } \right]+\sqnorm{ \sum_{m=1}^M \gv_{m} } \\
&\quad- \frac{2M}{N}\mathbb{E}_{\Omega^k}\left[\left\langle \sum_{m\in \Omega^k}\gv_{m}, \sum_{m=1}^M \gv_{m}\right\rangle \right]\\
&=\frac{M^2}{N^2} \mathbb{E}_{\Omega^k}\left[\sum_{m\in \Omega^k} \sqnorm{ \gv_{m} } \right]+
\frac{M^2}{N^2} \mathbb{E}_{\Omega^k}\left[\sum_{m\in \Omega^k}\sum_{m'\in \Omega^k, \neq m}\langle \gv_m,\gv_{m'}\rangle\right]\\
&\quad  - \sqnorm{ \sum_{m=1}^M \gv_{m} } .
\end{align*}
By computing the expectations on the right hand side, we finally get:
\begin{align*}
\Exp{ \sqnorm{ \sum_{m=1}^M \big(\mathcal{C}^k_m(\gv_m)-\gv_m\big) } }  &=\frac{M}{N} \sum_{m=1}^M\sqnorm{\gv_m}+
\frac{M(N-1)}{N(M-1)} \sum_{m=1}^M\sum_{m'=1, \neq m}^M\langle \gv_m,\gv_{m'}\rangle -
\sqnorm{ \sum_{m=1}^M \gv_m } \\
&=\frac{M}{N} \left(1-\frac{N-1}{M-1}\right) \sum_{m=1}^M\sqnorm{\gv_m}+
\left(\frac{M(N-1)}{N(M-1)}-1\right) 
\sqnorm{ \sum_{m=1}^M \gv_m } \\
&=\frac{M}{N} \frac{M-N}{M-1} \sum_{m=1}^M \sqnorm{ \gv_m }-\frac{M-N}{N(M-1)}
\sqnorm{ \sum_{m=1}^M \gv_m } .
\end{align*}
\hfill$\square$

\section{Proof of Proposition~\ref{prop2}}\label{appenb}

We have, for every $m\in[M]$ and $\gv_m\in\mathbb{R}^d$,
\begin{equation*}
\Exp{ \mathcal{C}_m'(\mathcal{C}_m(\gv_m))\;|\; \mathcal{C}_m(\gv_m) } =\mathcal{C}_m(\gv_m),
\end{equation*}
where the bar denotes conditional expectation, so that 
$
\Exp{ \mathcal{C}_m'(\mathcal{C}_m(\gv_m)) } =\gv_m,
$
and
\begin{equation*}
\Exp{ \sqnorm{ \mathcal{C}_m'(\mathcal{C}_m(\gv_m)) } \;|\; \mathcal{C}_m(\gv_m)} \leq (1+\omegac') \sqnorm{ \mathcal{C}_m(\gv_m) },
\end{equation*}
so that 
$
\Exp{ \sqnorm{ \mathcal{C}_m'(\mathcal{C}_m(\gv_m)) } }\leq (1+\omegac')\Exp{ \sqnorm{ \mathcal{C}_m(\gv_m) } } \leq (1+\omegac')(1+\omegac)\sqnorm{ \gv_m }.
$
Hence, 
\begin{equation*}
\Exp{ \sqnorm{ \mathcal{C}_m'(\mathcal{C}_m(\gv_m))-\gv_m } } \leq \left((1+\omegac')(1+\omegac)-1\right)\sqnorm{ \gv_m }.
\end{equation*}

Moreover, for every $\gv_m\in\mathbb{R}^d$, $m=1,\ldots,M$, 
\begin{equation*}
\Exp{ \sqnorm{ \frac{1}{M}\sum_{m=1}^M \mathcal{C}_m'\big(\mathcal{C}_m(\gv_m)\big) } \;\Big|\; \big(\mathcal{C}_m(\gv_m)\big)_{m=1}^M } \leq 
(1-\zeta')\sqnorm{ \frac{1}{M}\sum_{m=1}^M \mathcal{C}_m(\gv_m)}+\frac{\oma'}{M} \sum_{m=1}^M \sqnorm{\mathcal{C}_m(\gv_m)},
\end{equation*}
so that 
\begin{align*}
\Exp{ \sqnorm{ \frac{1}{M}\sum_{m=1}^M \mathcal{C}_m'\big(\mathcal{C}_m(\gv_m)\big) } } &  \leq 
(1-\zeta')\sqnorm{ \frac{1}{M}\sum_{m=1}^M \gv_m } +(1-\zeta') \Exp{ \sqnorm{ \frac{1}{M}\sum_{m=1}^M \big(\mathcal{C}_m(\gv_m)-\gv_m\big) } } \\
& \quad +\frac{\oma'}{M}(1+\omegac) \sum_{m=1}^M \sqnorm{ \gv_m }.
\end{align*}
Thus, if the $(\mathcal{C}_m)_{m=1}^M$ are mutually independent, 
\begin{align*}
\Exp{\sqnorm{ \frac{1}{M}\sum_{m=1}^M \mathcal{C}_m'\big(\mathcal{C}_m(\gv_m)\big) } } &  \leq 
(1-\zeta')\sqnorm{ \frac{1}{M}\sum_{m=1}^M \gv_m } + (1-\zeta')\frac{\omegac}{M^2}\sum_{m=1}^M \sqnorm{\gv_m}\\
&\quad+\frac{\oma'}{M}(1+\omegac) \sum_{m=1}^M \sqnorm{\gv_m},
\end{align*}
so that 
\begin{align*}
\Exp{ \sqnorm{ \frac{1}{M}\sum_{m=1}^M \Big(\mathcal{C}_m'\big(\mathcal{C}_m(\gv_m)\big)-\gv_m\Big) } } &  \leq 
\left(\frac{\omegac}{M}(1-\zeta')+\oma'(1+\omegac)
\right)\frac{1}{M}\sum_{m=1}^M \sqnorm{ \gv_m }\\
&\quad-\zeta'\sqnorm{ \frac{1}{M}\sum_{m=1}^M \gv_m } .
\end{align*}
\hfill$\square$

\section{Proof of Theorem~\ref{theo1}}\label{sec7}

Let us place ourselves in the conditions of Theorem~\ref{theo1}.
We define $h^\star \eqdef  \nabla F(x^\star)$ and $w^\star \eqdef x^\star-\gamma h^\star$. We have $x^\star=\prox_{\gamma R}(w^\star)$.

Let $k\in\mathbb{N}$. We have,  conditionally on $x^k$, $h^k$ and  $(h_m^k)_{m=1}^M$: $\mathbb{E}[\mathcal{R}^k(\tilde{x}^{k+1}-x^k)]=\tilde{x}^{k+1}-x^k$. Thus, using also \eqref{eq6} and the fact that $\nuc\leq \nuc'$,
\begin{align*}
\Exp{ \sqnorm{ x^{k+1}-x^\star } }& \leq \sqnorm{ (1-\rho)(x^{k}-x^\star)+\rho  (\tilde{x}^{k+1}-x^\star) } + \rho^2\nuc' \sqnorm{ \tilde{x}^{k+1}-x^{k} }\\
& \leq \left((1-\rho)^2+\rho^2\nuc'\right) \sqnorm{x^{k}-x^\star} + \rho^2 (1+\nuc') \sqnorm{ \tilde{x}^{k+1}- x^\star }\notag\\
&\quad + 2\rho\left(1-\rho(1+\nuc')\right) \big\langle x^{k}-x^\star,\tilde{x}^{k+1}-x^\star \big\rangle.\notag
\end{align*}
Thus, with $\rho=1/(1+\nuc')$,
\begin{equation}
\Exp{ \sqnorm{ x^{k+1}-x^\star } }
\leq \frac{\nuc'}{1+\nuc'} \sqnorm{ x^{k}-x^\star } + \frac{1}{1+\nuc'}\sqnorm{ \tilde{x}^{k+1}-x^\star }.\label{eq34}
\end{equation}
Moreover, using nonexpansiveness of the proximity operator and the fact that $\mathbb{E}[d^{k+1}]=\nabla F(x^k)-h^k$,
\begin{align*}
\Exp{ \sqnorm{ \tilde{x}^{k+1}-x^\star } } &\leq \Exp{ \sqnorm{ x^{k}-\gamma (d^{k+1}+h^{k})-w^\star } } \\
& = \Exp{ \sqnorm{ x^k-x^\star -\gamma(d^{k+1}+h^k-h^\star)} }\\
& = \Exp{ \sqnorm{ x^k-x^\star -\gamma\big(\nabla F(x^k)-\nabla F(x^\star)\big)} }+ \Exp{ \sqnorm{ d^{k+1}-\mathbb{E}[d^{k+1}]}}.
\end{align*}
We have $d^{k+1}=\frac{1}{M}\sum_{m=1}^M \mathcal{C}_m^k\big(\nabla F_m(x^k)-h_m^{k}\big)$. So, using \eqref{eqbo},
\begin{align*}
\Exp{ \sqnorm{ \tilde{x}^{k+1}-x^\star } }& \leq \sqnorm{ (\mathrm{Id}-\gamma \nabla F)x^k-(\mathrm{Id}-\gamma \nabla F)x^\star } + \frac{\gamma^2{\oma}}{M}\sum_{m=1}^M \sqnorm{ \nabla F_m(x^k)-h_m^k }\\
&\quad-\gamma^2\zeta \sqnorm{\nabla F(x^k)-h^k}\\
& = \sqnorm{ (\mathrm{Id}-\gamma \nabla F)x^k-(\mathrm{Id}-\gamma \nabla F)x^\star } 
+\gamma^2(\oma-\zeta) \sqnorm{\nabla F(x^k)-h^k}\\
&\quad + \frac{\gamma^2{\oma}}{M}\sum_{m=1}^M \sqnorm{ \nabla F_m(x^k)-h_m^k - \nabla F(x^k)+h^k},
\end{align*}
where we used the fact that for every vectors $\gv_m$, $m=1,\ldots,M$, $\frac{1}{M}\sum_{m=1}^{M} \|\gv_m\|^2=\frac{1}{M}\sum_{m=1}^{M} \|\gv_m-\gv\|^2+\|\gv\|^2$, where $\gv=\frac{1}{M}\sum_{m=1}^{M}\gv_m$. Now, we will use the fact that $\oma-\zeta\geq 0$ and the Peter--Paul inequality, according to which, for every $\gv\in\mathbb{R}^d$ and $\gv'\in\mathbb{R}^d$,  $\sqnorm{\gv+\gv'}\leq \left(1+\frac{1}{\bc}\right)\sqnorm{\gv} + (1+\bc)\sqnorm{ \gv'}$. Thus,\\
\begin{align*}
\Exp{ \sqnorm{ \tilde{x}^{k+1}-x^\star } } 
& \leq \sqnorm{ (\mathrm{Id}-\gamma \nabla F)x^k-(\mathrm{Id}-\gamma \nabla F)x^\star } 
+\left(1+\frac{1}{\bc}\right)  \gamma^2(\oma-\zeta) \sqnorm{ h^k-h^\star}\\
&\quad+ (1+\bc) \gamma^2(\oma-\zeta) \sqnorm{ \nabla F(x^k)-\nabla F(x^\star)}\\
&  \quad+ \left(1+\frac{1}{\bc}\right)\gamma^2{\oma}\frac{1}{M}\sum_{m=1}^M \sqnorm{ h_m^k- h_m^\star -h^k+h^\star}\\
&  \quad+ (1+\bc)\gamma^2{\oma}\frac{1}{M}\sum_{m=1}^M \sqnorm{ \nabla F_m(x^k)-\nabla F_m(x^\star) -\nabla F(x^k)+\nabla F(x^\star) }\\
& \leq \sqnorm{ (\mathrm{Id}-\gamma \nabla F)x^k-(\mathrm{Id}-\gamma \nabla F)x^\star } - (1+\bc) \gamma^2\zeta \sqnorm{ \nabla F(x^k)-\nabla F(x^\star)}\\
&  \quad+ \left(1+\frac{1}{\bc}\right)\gamma^2{\oma}\frac{1}{M}\sum_{m=1}^M \sqnorm{ h_m^k- h_m^\star }\\
&  \quad+ (1+\bc)\gamma^2{\oma}\frac{1}{M}\sum_{m=1}^M \sqnorm{ \nabla F_m(x^k)-\nabla F_m(x^\star)  }\\
& = \sqnorm{x^k-x^\star}  -2\gamma  \big\langle x^k-x^\star, \nabla F(x^k)-\nabla F(x^\star) \big\rangle \\
&\quad +\gamma^2\big(1- (1+\bc) \zeta\big) \sqnorm{ \nabla F(x^k)-\nabla F(x^\star)}\\
&  \quad+ \left(1+\frac{1}{\bc}\right)\gamma^2{\oma}\frac{1}{M}\sum_{m=1}^M \sqnorm{ h_m^k- h_m^\star }\\
&  \quad+ (1+\bc)\gamma^2{\oma}\frac{1}{M}\sum_{m=1}^M \sqnorm{ \nabla F_m(x^k)-\nabla F_m(x^\star)  }\\
& \leq \sqnorm{x^k-x^\star}  -2\gamma  \big\langle x^k-x^\star, \nabla F(x^k)-\nabla F(x^\star) \big\rangle \\
&\quad + \left(1+\frac{1}{\bc}\right)\gamma^2{\oma}\frac{1}{M}\sum_{m=1}^M \sqnorm{ h_m^k- h_m^\star }\\
&  \quad+ \gamma^2\Big(\max\big(1-(1+\bc) \zeta,0\big)+(1+\bc){\oma}\Big)\frac{1}{M}\sum_{m=1}^M \sqnorm{ \nabla F_m(x^k)-h_m^\star },
\end{align*}
where we used the fact that if the constant in front of $\|\nabla F(x^k)-\nabla F(x^\star)\|^2$ is negative, we can ignore this term, whereas if it positive, we have to upper bound it.

In addition, 
\begin{align*}
\langle x^k-x^\star, \nabla F(x^k)-\nabla F(x^\star)\rangle &= \eta \langle x^k-x^\star, \nabla F(x^k)-\nabla F(x^\star)\rangle \\
& \quad+  (1-\eta) \frac{1}{M}\sum_{m=1}^M \langle x^k-x^\star,\nabla F_m(x^k)-\nabla F_m(x^\star)\rangle.
\end{align*}
By $\mmuF$-strong convexity of $F$, $\nabla F-\mmuF \mathrm{Id}$ is monotone, so that $\langle x^k-x^\star, \nabla F(x^k)-\nabla F(x^\star)\rangle \geq \mmuF \|x^k-x^\star\|^2$. Also, by cocoercivity of the gradient, for every $m\in[M]$,  $ \langle x^k-x^\star, \nabla F_m(x^k)-\nabla F_m(x^\star)\rangle\geq \frac{1}{L} \| \nabla F_m(x^k)-\nabla F_m(x^\star)\|^2$. 
So,
\begin{align*}
\langle x^k-x^\star, \nabla F(x^k)-\nabla F(x^\star)\rangle& \geq \eta \mmuF  \sqnorm{x^k-x^\star}+ (1-\eta) \frac{1}{L}  \frac{1}{M}\sum_{m=1}^M \sqnorm{ \nabla F_m(x^k)-\nabla F_m(x^\star) }.\end{align*}
Hence, using the definition of $\ac$,
\begin{align*}
\Exp{ \sqnorm{ \tilde{x}^{k+1}-x^\star } }& \leq \left(
1-2\gamma \eta \mmuF 
\right)\sqnorm{x^k-x^\star} + \left(1+\frac{1}{\bc}\right)\gamma^2{\oma}\frac{1}{M}\sum_{m=1}^M \sqnorm{h_m^k- h_m^\star}\\
&\hspace{-1cm}+ \left(
\gamma^2\Big(\ac+(1+\bc){\oma}\Big)-2\gamma(1-\eta) \frac{1}{L}\right)\!\frac{1}{M}\sum_{m=1}^M \sqnorm{ \nabla F_m(x^k)-h_m^\star }
\end{align*}
and, by combination with \eqref{eq34},
\begin{align*}
\Exp{ \sqnorm{ x^{k+1}-x^\star } }& \leq \left(
1-\frac{2\gamma\eta \mmuF}{1+\nuc'}
\right)\sqnorm{x^k-x^\star}+ \left(1+\frac{1}{\bc}\right)\frac{\gamma^2{\oma}}{1+\nuc'}\frac{1}{M}\sum_{m=1}^M \sqnorm{h_m^k- h_m^\star}\\
&  \quad+\frac{1}{1+\nuc'} \left(
\gamma^2\big(\ac+(1+\bc){\oma}\big)-2\gamma(1-\eta) \frac{1}{L}\right)\!\frac{1}{M}\sum_{m=1}^M \sqnorm{ \nabla F_m(x^k)-h_m^\star }.
\end{align*}

On the other hand, 
conditionally on $x^k$, $h^k$, and $(h_m^k)_{m=1}^M$, we have, for every $m\in[M]$,
\begin{align*}
\Exp{ \sqnorm{ h_m^{k+1}-h_m^\star } }& \leq \sqnorm{ (1-\lambda)(h_m^{k}-h_m^\star)+\lambda  \big(\nabla F_m(x^k)-h_m^\star\big) } + \lambda^2\chic' \sqnorm{ \nabla F_m(x^k)-h_m^{k}}\\
& \leq \left((1-\lambda)^2+\lambda^2\chic'\right) \sqnorm{ h_m^{k}-h_m^\star } + \lambda^2 (1+\chic')\sqnorm{ \nabla F_m(x^k)-h_m^\star } \notag\\
&\quad + 2\lambda\big(1-\lambda(1+\chic')\big) \big\langle h_m^{k}-h_m^\star,\nabla F_m(x^k)-h_m^\star\big\rangle.\notag
\end{align*}
Thus, with $\lambda=1/(1+\chic')$,
\begin{align*}
\Exp{ \sqnorm{ h_m^{k+1}-h_m^\star } }
&\leq \frac{\chic'}{1+\chic'} \sqnorm{ h_m^{k}-h_m^\star } + \frac{1}{1+\chic'}\sqnorm{ \nabla F_m(x^k)-h_m^\star }.\notag
\end{align*}

Thus, 
conditionally on $x^k$, $h^k$, and  $(h_m^k)_{m=1}^M$,
\begin{align*}
\Exp{ \Psi^{k+1} }& \leq \left(
1-\frac{2\gamma\eta\mmuF}{1+\nuc'} 
\right)\sqnorm{x^k-x^\star}+  \frac{1+\bc^2\chic'}{\bc^2(1+\chic')}(\bc^2+\bc)\gamma^2{\oma}\frac{1+\chic'}{1+\nuc'}\frac{1}{M}\sum_{m=1}^M \sqnorm{h_m^k- h_m^\star}\\
&  \quad+\frac{1}{1+\nuc'} \left(
\gamma^2\big(\ac+(1+\bc)^2{\oma}\big)-2\gamma(1-\eta) \frac{1}{L}\right)\frac{1}{M}\sum_{m=1}^M \sqnorm{ \nabla F_m(x^k)-h_m^\star }.
\end{align*}
By definition of $\eta$, $\gamma =  \frac{2(1-\eta)}{L}\frac{1}{\ac+(1+\bc)^2{\oma}}$, so that the last term above is zero and
\begin{align*}
\Exp{ \Psi^{k+1} }& \leq \left(
1-\frac{2\gamma\eta\mmuF}{1+\nuc'} 
\right)\sqnorm{x^k-x^\star}+  \frac{1+\bc^2\chic'}{\bc^2(1+\chic')}(\bc^2+\bc)\gamma^2{\oma}\frac{1+\chic'}{1+\nuc'}\frac{1}{M}\sum_{m=1}^M \sqnorm{h_m^k- h_m^\star}\\
&  \leq c \Psi^k,
\end{align*}
where
\begin{align*}
c = \max\left\{ 1-\frac{2\gamma\eta\mmuF}{1+\nuc'}
,{  \frac{\bc^{-2}+\chic'}{1+\chic'}}\right \}.
\end{align*}
Since $\bc>1$, we have $c<1$. 

Finally, iterating the tower rule on the conditional expectations, we have, for every $k\geq 0$, 
\begin{equation*}
\Exp{ \Psi^{k} } \leq c^k \Psi^{0}.
\end{equation*}
\hfill$\square$

\section{Experiments}\label{secex}

\begin{figure}[t]
\centering
\includegraphics[scale=0.75]{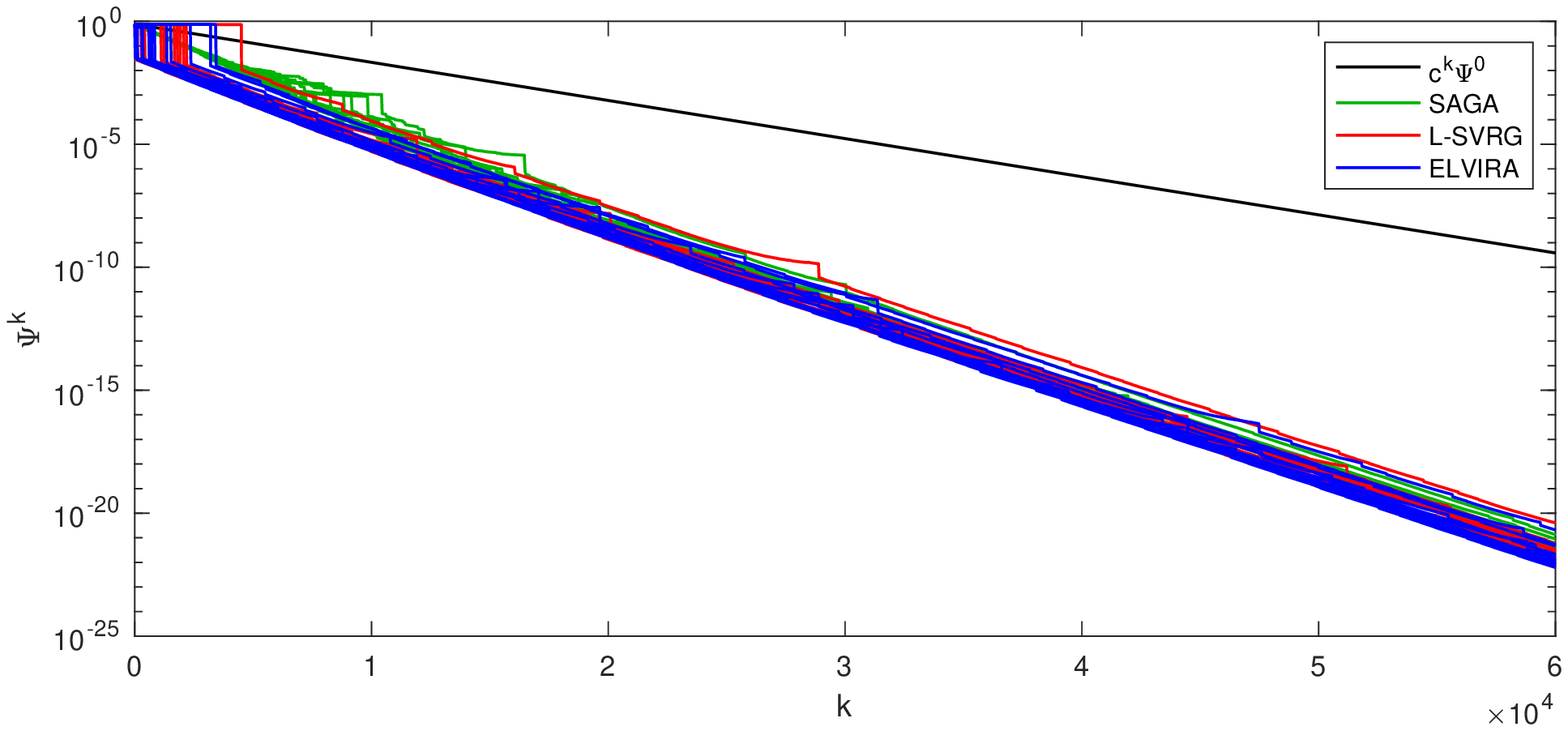}
\caption{Convergence plots for a synthetic experiment with quadratic functions, with 15 different runs for each stochastic algorithm. }
\end{figure}

We compare SAGA, L-SVRG and ELVIRA on the same synthetic problem of minimizing over $\mathbb{R}^d$ the average of $M=1000$ functions $F_m$, with $d=100$; that is, Problem \eqref{eqa1} with $R=0$. Every function $F_m$ is quadratic: $F_m : x \mapsto  \frac{1}{2}\|A_m x -b_m\|^2$ for some matrix $A_m$ of size $d' \times d$ and vector $b_m \in \mathbb{R}^{d'}$, all made of independent random values drawn from the uniform distribution in $[0,1]$, with $d'=5$. Since $d'< d$, none of the $F_m$ is strongly convex, but their average $F$ is $\mmuF$-strongly convex, with $\mmuF\approx 0.3$. Every $F_m$ is $L$-smooth, with $L=\max_{m=1,\ldots,M} \|A_m^*A_m\|\approx 153$. We choose $b=1.4$ so that the 2 terms in the rate $c$ are equal and $\approx 0.9996$ and we set $\gamma=\frac{1}{L(1+b)^2}$ in the 3 algorithms. In L-SVRG and ELVIRA, $N=1$ and $p=\frac{1}{M}$. Then the Lyapunov function $\Psi^k$ is the same for the 3 algorithms, as well as the rate $c\approx 0.9996$. We show the upper bound $c^k\Psi^0$ in black in Figure~1. The solutions $x^\star$ and $h_m^\star$ were computed to machine precision by running SAGA with $10^6$ iterations. The value of $\Psi^k$ with respect to $k$ is shown in Figure~1 for the 3 algorithms, for 15 different runs of each algorithm. We can observe that the algorithms converge linearly, as proved by our convergence results, with an empirical convergence rate better than the upper bound. The 3 algorithms have rather similar convergence profiles, with convergence slightly slower for SAGA, ELVIRA performing best, with less choppy curves, and L-SVRG in between. The convergence is shown with respect to the iteration index $k$, but we should keep in mind that SAGA has 1 gradient evaluation per iteration, whereas in average L-SVRG and ELVIRA have 3. But SAGA needs to store all the vectors $h_m$, while L-SVRG and ELVIRA do not need such memory occupation.

\end{document}